\documentclass[11pt,a4paper,reqno]{amsart}
\usepackage[english]{babel}
\usepackage[T1]{fontenc}
\usepackage{verbatim}
\usepackage{palatino}
\usepackage{amsmath}
\usepackage{mathabx}
\usepackage{amssymb}
\usepackage{amsthm}
\usepackage{amsfonts}
\usepackage{graphicx}
\usepackage{esint}
\usepackage{color}
\usepackage{mathtools}
\usepackage{overpic}

\usepackage[colorlinks,linkcolor=blue,citecolor=blue]{hyperref}
\author{Tuomas Orponen}
\title{Large cliques in extremal incidence configurations}
\address{Department of Mathematics and Statistics\\ University of Jyv\"askyl\"a,
	P.O. Box 35 (MaD)\\
	FI-40014 University of Jyv\"askyl\"a\\
	Finland} \email{tuomas.t.orponen@jyu.fi}
\author[G. Yi]{Guangzeng Yi}
\address{$\ast$Department of Mathematics and Statistics\\ University of Jyv\"askyl\"a,
P.O. Box 35 (MaD)\\
FI-40014 University of Jyv\"askyl\"a\\
Finland}
\email{guangzeng.m.yi@jyu.fi}
\date{\today}
\thanks{T.O. is supported by the Research Council of Finland via the project \emph{Approximate incidence geometry}, grant no. 355453. Both T.O. and G.Y. are supported by the European Research Council (ERC) under the European Union’s Horizon Europe research and innovation programme (grant agreement No 101087499). }

\newcommand{\R}{\mathbb{R}}

\newcommand{\N}{\mathbb{N}}

\newcommand{\diam}{\operatorname{diam}}

\newcommand{\dist}{\operatorname{dist}}

\def\Barint_#1{\mathchoice
{\mathop{\vrule width 6pt height 3 pt depth -2.5pt
	\kern -8pt \intop}\nolimits_{#1}}%
{\mathop{\vrule width 5pt height 3 pt depth -2.6pt
	\kern -6pt \intop}\nolimits_{#1}}%
{\mathop{\vrule width 5pt height 3 pt depth -2.6pt
	\kern -6pt \intop}\nolimits_{#1}}%
{\mathop{\vrule width 5pt height 3 pt depth -2.6pt
	\kern -6pt \intop}\nolimits_{#1}}}

\numberwithin{equation}{section}

\theoremstyle{plain}
\newtheorem{thm}[equation]{Theorem}
\newtheorem*{"thm"}{"Theorem"}

\newtheorem{lemma}[equation]{Lemma}

\newtheorem{cor}[equation]{Corollary}
\newtheorem{proposition}[equation]{Proposition}

\theoremstyle{definition}

\newtheorem{definition}[equation]{Definition}

\theoremstyle{remark}
\newtheorem{remark}[equation]{Remark}

\newcommand{\nref}[1]{(\hyperref[#1]{#1})}

\DeclareMathSymbol{\intop}  {\mathop}{mathx}{"B3}

\begin{document}
	
\begin{abstract} Let $P \subset \R^{2}$ be a Katz-Tao $(\delta,s)$-set, and let $\mathcal{L}$ be a Katz-Tao $(\delta,t)$-set of lines in $\R^{2}$. A recent result of Fu and Ren gives a sharp upper bound for the $\delta$-covering number of the set of incidences $\mathcal{I}(P,\mathcal{L}) = \{(p,\ell) \in P \times \mathcal{L} : p \in \ell\}$. In fact, for $s,t \in (0,1]$,
$$ |\mathcal{I}(P,\mathcal{L})|_{\delta} \lesssim_{\epsilon} \delta^{-\epsilon -f(s,t)}, \qquad \epsilon > 0,$$
where $f(s,t) = (s^{2} + st + t^{2})/(s + t)$. For $s,t \in (0,1]$, we characterise the near-extremal configurations $P \times \mathcal{L}$ of this inequality: we show that if $|\mathcal{I}(P,\mathcal{L})|_{\delta} \approx \delta^{-f(s,t)}$, then $P \times \mathcal{L}$ contains "cliques" $P' \times \mathcal{L}'$ satisfying $|\mathcal{I}(P',\mathcal{L}')|_{\delta} \approx |P'|_{\delta}|\mathcal{L}'|_{\delta}$, 
$$|P'|_{\delta} \approx \delta^{-s^{2}/(s + t)} \quad \text{and} \quad |\mathcal{L}'|_{\delta} \approx \delta^{-t^{2}/(s + t)}.$$ \end{abstract}
	
\maketitle

\setcounter{tocdepth}{1}
\tableofcontents

\section{Introduction}\label{section 1}

This paper studies the $\delta$-covering number of incidences between sets of points and lines in $\R^{2}$. Let $P \subset \R^{2}$ and $\mathcal{L} \subset \mathcal{A}(2)$, where $\mathcal{A}(2)$ is the space of all (affine) lines in $\R^{2}$. The \emph{incidences} between $P$ and $\mathcal{L}$ are the pairs
\begin{displaymath} \mathcal{I}(P,\mathcal{L}) = \{(p,\ell) \in P \times \mathcal{L} : p \in \ell\}. \end{displaymath}
We equip $\R^{d}$ with the Euclidean norm $|\cdot|$, and $\mathcal{A}(2)$ with the metric 
\begin{displaymath} d_{\mathcal{A}(2)}(\ell_{1},\ell_{2}) = \|\pi_{L_{1}} - \pi_{L_{2}}\| + |a_{1} - a_{2}|, \end{displaymath}
whenever $\ell_{j} = L_{j} + a_{j}$, and $L_{j}$ is the $1$-dimensional subspace parallel to $\ell_{j}$. If $P \subset \R^{2}$, $\mathcal{L} \subset \mathcal{A}(2)$, and $\delta > 0$, the notations $|P|_{\delta}$ and $|\mathcal{L}|_{\delta}$ refer to the $\delta$-covering numbers relative to the Euclidean and $d_{\mathcal{A}(2)}$-metrics, respectively. For $\mathcal{I} \subset \R^{2} \times \mathcal{A}(2)$, the notation $|\mathcal{I}|_{\delta}$ refers to the $\delta$-covering number in the metric $d((x,\ell),(x',\ell')) = \max\{|x -x'|,d_{\mathcal{A}(2)}(\ell,\ell')\}$.

 The following notion is central to this paper:
\begin{definition}[$(\delta,\theta)$-clique] For $\delta \in 2^{-\N}$ and $\theta \in [0,1]$, a \emph{$(\delta,\theta)$-clique} is a pair $P \times \mathcal{L} \subset \R^{2} \times \mathcal{A}(2)$ with $|\mathcal{I}(P,\mathcal{L})|_{\delta} \geq \theta |P|_{\delta}|\mathcal{L}|_{\delta}$. \end{definition}

The main purpose will be, roughly speaking, to show that if $|\mathcal{I}(P,\mathcal{L})|_{\delta}$ is "extremal", then $P \times \mathcal{L}$ needs to contain large $(\delta,\theta)$-sub-cliques with $\theta \approx 1$. What is meant by "extremal"? A typical result in ($\delta$-discretised) incidence geometry gives an upper bound for $|\mathcal{I}(P,\mathcal{L})|_{\delta}$, provided that $P$ and $\mathcal{L}$ satisfy some non-concentration conditions. Then, an extremal configuration is a pair $P \times \mathcal{L}$ which satisfies these non-concentration conditions, and such that $|\mathcal{I}(P,\mathcal{L})|_{\delta}$ (nearly) realises the upper bound. In particular, the definition of "extremal" depends on the choice of non-concentration conditions.

We focus on the following definition, originally introduced by Katz and Tao \cite{KT01}:
\begin{definition}[Katz-Tao $(\delta,s,C)$-set]\label{def:katzTaoSet}
Let $P\subset\mathbb R^d$ be a bounded set, $d\geq 2$. Let $\delta\in (0,1]$, $0\leq s\leq d$ and $C>0$. We say that $P$ is a \emph{Katz-Tao $(\delta,s,C)$-set} if
\begin{equation}\label{def 1}
|P\cap B(x,r)|_\delta\leq C\left(\tfrac{r}{\delta}\right)^s, \qquad x \in \R^{2}, \, \delta \leq r \leq 1. 
\end{equation}
If $\mathcal{P} \subset \mathcal{D}_{\delta}(\R^{2})$, we say that $\mathcal{P}$ is a Katz-Tao $(\delta,s,C)$-set if $P := \cup \mathcal{P}$ satisfies \eqref{def 1}.

A line family $\mathcal{L} \subset \mathcal{A}(2)$ is called a \emph{Katz-Tao $(\delta,s,C)$-set} if 
\begin{displaymath} |\mathcal{L} \cap B_{\mathcal{A}(2)}(\ell,r)|_{\delta} \leq C\left(\tfrac{r}{\delta} \right)^s, \qquad \ell \in \mathcal{A}(2), \, \delta \leq r \leq 1, \end{displaymath}
Here $B_{\mathcal{A}(2)}(\ell,r)$ refers to a ball in the metric $d_{\mathcal{A}(2)}$.

A $(\delta,s,C)$-set (of points or lines) is called a $(\delta,s)$-set if the value of the constant $C > 0$ is irrelevant. 
\end{definition}

\begin{remark} Note that a Katz-Tao $(\delta,s)$-set of points or lines may well be infinite. A reasonable intuition is that $P$ is a finite union of $\delta$-discs or $\delta$-squares, whereas $\mathcal{L}$ is the collection of lines foliating a finite union of $\delta$-tubes or \emph{dyadic $\delta$-tubes} (see Definition \ref{def:dyadicTube}). \end{remark}

If $P \subset \R^{2}$ is a Katz-Tao $(\delta,s)$-set, and $\mathcal{L} \subset \mathcal{A}(2)$ is a Katz-Tao $(\delta,t)$-set of lines, the sharp upper bound for $|\mathcal{I}(P,\mathcal{L})|_{\delta}$ was recently established by Fu and Ren \cite{fu2022incidence}: 
\begin{thm}\label{t:FuRen}
Let $s, t\in (0,1]$ and $K_{P}, K_{\mathcal{L}}\geq 1$. For every $\epsilon>0$, there exists a constant $C=C(\epsilon,K_{P}, K_{\mathcal{L}})$ such that the following holds. Assume $P \subset [0,1]^{2}$ is a Katz-Tao $(\delta,s, K_{P})$-set and $\mathcal{L} \subset \mathcal{A}(2)$ is a Katz-Tao $(\delta,t,K_{\mathcal{L}})$- set. Then 
\[|\mathcal{I}(P,\mathcal{L})|_{\delta} \leq C \delta^{-\epsilon-f(s,t)},\]
where $f(s,t)=\frac{s^2+st+t^2}{s+t}$. Moreover, this bound is sharp up to $C\delta^{-\epsilon}$.
\end{thm}

In fact, Fu and Ren established a sharp bound for all $s \in (0,2]$ and $t \in (0,2]$ (the definition of $f$ is then piece-wise, depending on the range of $s,t$). In this paper we restrict attention to the cases $s,t \in (0,1]$.

\begin{remark}\label{rem3} The result of Fu and Ren was originally stated slightly differently. In \cite{fu2022incidence}, the set $P = \cup \mathcal{B}$ is a finite union of $\delta$-discs and $\mathcal{L} = \cup \mathcal{T}$ is a finite union of $\delta$-tubes. The incidences are defined in \cite{fu2022incidence} as $\overline{\mathcal{I}}(\mathcal{B},\mathcal{T}) = \{(B,T) \in\mathcal{B}\times \mathcal{T}: B \cap T \neq \emptyset\}$. Under the hypotheses of Theorem \ref{t:FuRen}, the authors established the inequality $|\overline{\mathcal{I}}(\mathcal{B},\mathcal{T})| \lesssim_{\epsilon} \delta^{-\epsilon - f(s,t)}$. We will check in Remark \ref{rem2} that Theorem \ref{t:FuRen} follows, as stated, from its original version in \cite{fu2022incidence}.  \end{remark}

\begin{remark} It took a few attempts before we ended up studying the quantity $|\mathcal{I}(P,\mathcal{L})|_{\delta}$. Other, perhaps more obvious, alternatives would have been to consider incidences between $\delta$-balls and ordinary tubes (as in \cite{fu2022incidence}), or dyadic $\delta$-squares and dyadic $\delta$-tubes (as in e.g. \cite{OS23}). However, if our main result, Theorem \ref{main}, was stated for any one of these "standard" choices, it seemed hard to deduce the other "standard" choices as corollaries. 

We will not explain the problem in detail, but it has to do with the following phenomenon. Assume that $\mathcal{B}$ and $\mathcal{T}$ are $\delta$-neighbourhoods of points and lines in some metric of $\R^{2}$. Then, the cardinality $|\overline{\mathcal{I}}(\mathcal{B},\mathcal{T})| = |\{(B,T) \in \mathcal{B} \times \mathcal{T} : B \cap T \neq \emptyset\}|$ is far from being (roughly) invariant under bi-Lipschitz transformations of that metric. For example, there are configurations where a large family of Euclidean $\delta$-tubes narrowly avoids a large family of Euclidean $\delta$-balls, but the collinear $(2\delta)$-tubes already hit all the concentric $(2\delta)$-balls. In contrast, the $\delta$-covering number $|\mathcal{I}(P,\mathcal{L})|_{\delta}$ only changes by a constant if the metrics of $\R^{2}$ and $\mathcal{A}(2)$ are replaced by bi-Lipschitz equivalent ones.

Theorem \ref{main} is formulated in terms of $|\mathcal{I}(P,\mathcal{L})|_{\delta}$ to make it more robust. Now it actually implies other (possibly more) "standard" versions as corollaries. We mention one concrete example in Remark \ref{rem4}, but omit the straightforward details. \end{remark}

The bound in Theorem \ref{t:FuRen} is sharp, but weaker than the \emph{Szemer\'edi-Trotter} bound on incidences between families of points and lines. If $P \subset \R^{2}$ is a finite set, and $\mathcal{L}$ is a finite set of lines, Szemer\'edi and Trotter \cite{MR729791} in 1983 established the following:
\begin{equation}\label{st} |\mathcal{I}(P,\mathcal{L})| \lesssim |P|^{2/3}|\mathcal{L}|^{2/3} + |P| + |\mathcal{L}|. \end{equation} 
For example, if $s = t \in (0,1]$, then Theorem \ref{t:FuRen} gives $|\mathcal{I}(P,\mathcal{L})|_{\delta} \lesssim_{\epsilon} \delta^{-3s/2 - \epsilon}$, whereas a formal application of \eqref{st} would predict that $|\mathcal{I}(P,\mathcal{L})|_{\delta} \lesssim \delta^{-4s/3}$.
\begin{figure}[h!]
\begin{center}
\begin{overpic}[scale = 0.6]{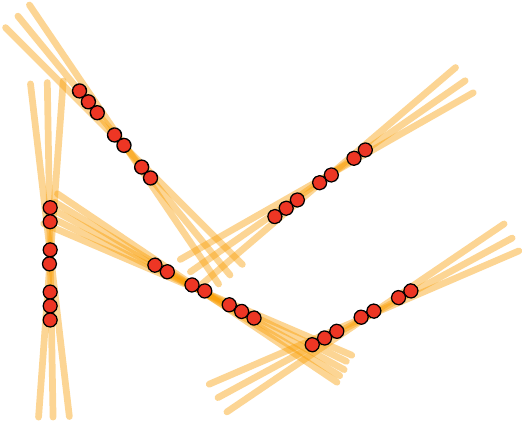}
\end{overpic}
\caption{A pair $P \cap \mathcal{L}$ admitting a decomposition into large $\delta$-sub-cliques.}\label{fig3}
\end{center}
\end{figure}

In the sharp examples provided by Fu and Ren to Theorem \ref{t:FuRen}, the large number of incidences is due to many large $(\delta,1)$-cliques. In fact, in these examples both $P$ and $\mathcal{L}$ are partitioned as $P = P_{1} \cup \ldots \cup P_{N}$ and $\mathcal{L} = \mathcal{L}_{1} \cup \ldots \cup \mathcal{L}_{N}$ in such a way that $P_{j} \times \mathcal{L}_{j}$ is a $(\delta,1)$-clique for all $1 \leq j \leq N$, see Figure \ref{fig3} for an illustration.

The Katz-Tao conditions impose the following restriction on $(\delta,1)$-cliques: if $P \times \mathcal{L}$ is a $(\delta,1)$-clique, where $P$ is a Katz-Tao $(\delta,s)$-set, and $\mathcal{L} \subset \mathcal{A}(2)$ is a Katz-Tao $(\delta,t)$-set, then $|P|_{\delta}^{t}|\mathcal{L}|_{\delta}^{s} \lesssim \delta^{-st}$. This follows from Proposition \ref{prop4} below. Optimising under this constraint, one finds that the most $\delta$-incidences are generated by a $(\delta,1)$-clique decomposition where 
\begin{equation}\label{form37} |P_{j}|_{\delta} \equiv \delta^{-s^{2}/(s + t)} \quad \text{and} \quad |\mathcal{L}_{j}|_{\delta} \equiv \delta^{-t^{2}/(s + t)}. \end{equation}
Indeed, the number of incidences in such a configuration matches the upper bound in Theorem \ref{t:FuRen}, up to the constant $C\delta^{-\epsilon}$.

Our main result shows that any (near-)extremal configuration for Theorem \ref{t:FuRen} must contain cliques of (nearly) the size \eqref{form37}:

\begin{thm}\label{main}
For every $u\in(0,1]$ and $s,t \in (0,1]$, there exist $\delta_0=\delta_0(s,t,u)>0$ and $\epsilon=\epsilon(s,t,u)>0$ such that the following holds for any $\delta\in (0, \delta_0]$. Write $f(s,t)=\tfrac{s^2+st+t^2}{s+t}$. Let $P \subset [0,1]^{2}$ be a Katz-Tao $(\delta,s,\delta^{-\epsilon})$-set, and let $\mathcal{L} \subset \mathcal{A}(2)$ be a Katz-Tao $(\delta,t,\delta^{-\epsilon})$-set. If
\begin{equation}\label{form36}
|\mathcal{I}(P,\mathcal{L})|_{\delta} \geq \delta^{\epsilon-f(s,t)},
\end{equation}
then there exists a $(\delta,\delta^{u})$-clique $P' \times \mathcal{L}' \subset P \times \mathcal{L}$ with 
\begin{equation}\label{form41} |P'|_{\delta} \geq \delta^{u - s^{2}/(s + t)} \quad \text{and} \quad |\mathcal{L}'|_{\delta} \geq \delta^{u -t^{2}/(s + t)}. \end{equation}
\end{thm}

\begin{remark}\label{rem5} The proof of Theorem \ref{main} (see \eqref{clique}) yields a $(\delta,\delta^{u})$-clique $P' \times \mathcal{L}'$ of the form $P' = P \cap Q$ and $\mathcal{L}' = \mathcal{L} \cap \mathcal{T}'$, where $Q$ is a dyadic square of some side-length $\Delta \in [\delta,1]$, and $\mathcal{T}' \subset \mathcal{T}^{\delta}$ is a family of dyadic $\delta$-tubes (see Definition \ref{def:dyadicTube}). Here $\mathcal{L} \cap \mathcal{T}'$ refers to the set of lines in $\mathcal{L}$ contained in at least one element of $\mathcal{T}'$. \end{remark}

\begin{remark}\label{rem4} Let $s,t \in (0,1]$ and $u > 0$. In the notation of Remark \ref{rem3}, assume that $\mathcal{B}$ is a Katz-Tao $(\delta,s)$-set of $\delta$-discs, and $\mathcal{T}$ is a Katz-Tao $(\delta,t)$-set of $\delta$-tubes satisfying $|\overline{\mathcal{I}}(\mathcal{B},\mathcal{L})| \geq \delta^{\epsilon - f(s,t)}$. Then, if $\delta > 0$ is sufficiently small in terms of $s,t,u$, there exist $\mathcal{B}' \subset \mathcal{B}$ and $\mathcal{T}' \subset \mathcal{T}$ such that $|\overline{\mathcal{I}}(\mathcal{B}',\mathcal{T}')| \geq \delta^{u}|\mathcal{B}'||\mathcal{T}'|$, and $\mathcal{B}',\mathcal{T}'$ satisfy \eqref{form41}. This follows easily from Theorem \ref{main} applied to $P = \cup \mathcal{B}$ and $\mathcal{L} = \{\ell \in \mathcal{A}(2) : \ell \subset T \text{ for some } T \in \mathcal{T}\}$.\end{remark}

While Theorem \ref{main} only states the existence of a single $(\delta,\delta^{u})$-clique, a formal "exhaustion argument" shows that there are many $(\delta,\delta^{u})$-cliques: they are indeed responsible for a major part of the incidences.

\begin{cor}\label{cor3} Under the hypotheses of Theorem \ref{main}, there exists a list
\begin{displaymath} (P_{1} \times \mathcal{L}_{1}),\ldots,(P_{n} \times \mathcal{L}_{n}) \subset P \times \mathcal{L} \end{displaymath}
of $(\delta,\delta^{u})$-cliques satisfying \eqref{form41}, with the sets $\mathcal{D}_{\delta}(P_{j})$ disjoint, and $\sum_{j} |\mathcal{I}(P_{j},\mathcal{L}_{j})|_{\delta} \geq \delta^{u - f(s,t)}$. \end{cor}

Does Corollary \ref{cor3} imply that the only configurations $P \times \mathcal{L}$ satisfying \eqref{form36} must contain a sub-configuration as in Figure \ref{fig3}? In other words, do $(\delta,\delta^{u})$-cliques resemble the "sheaves" from Figure \ref{fig3}? The answer is affirmative, up to passing to further subsets. This follows from the next proposition, combined with the subsequent remark:

\begin{proposition}\label{prop4} There exists an absolute constant $C \geq 1$ such that the following holds. Let $P \times \mathcal{L}$ be a $(\delta,\theta)$-clique. Then, there exists a rectangle $R \subset \R^{2}$ of dimensions $C(\delta \times \Delta)$, where $\Delta \in [\delta,2]$, such that
\begin{equation}\label{form40} |P \cap R|_{\delta} \gtrsim \theta^{2}|P|_{\delta} \quad \text{and} \quad |\{\ell \in \mathcal{L} : R \subset [\ell]_{C\delta}\}|_{\delta} \gtrapprox \theta^{4}|\mathcal{L}|_{\delta}. \end{equation}
Here $[\ell]_{C\delta}$ is the $C\delta$-neighbourhood of $\ell$. In particular: if $P$ is a Katz-Tao $(\delta,s)$-set, and $\mathcal{L}$ is a Katz-Tao $(\delta,t)$-set, then $|P|_{\delta}^{t}|\mathcal{L}|_{\delta}^{s} \lessapprox \theta^{-6}\delta^{-st}$. \end{proposition}

In Proposition \ref{prop4}, the notation $A \lessapprox B$ means that $A \leq C(\log(1/\delta))^{C}B$ for some absolute constant $C \geq 1$. This notation will serve various purposes in the paper, and we will always define it separately. 

\begin{remark} When Proposition \ref{prop4} is applied to the $(\delta,\delta^{u})$-clique $P' \times \mathcal{L}'$ in Theorem \ref{main}, the diameter $\Delta$ of the rectangle $R$ is (almost) uniquely determined. In fact,
\begin{equation}\label{form39} \diam(R) \sim \Delta \approx \delta^{t/(s + t)}. \end{equation}
The "$\approx$" and "$\lessapprox$" notations in this remark is allowed to hide factors of the form $\delta^{-C\epsilon}$ and $\delta^{-Cu}$. To verify \eqref{form39}, we first deduce from the lower bound $|P'|_{\delta} \gtrapprox \delta^{- s^{2}/(s + t)}$ combined with the Katz-Tao $(\delta,s)$-set condition of $P$ that
\begin{displaymath} (\Delta/\delta)^{s} \gtrapprox |P' \cap R|_{\delta} \gtrapprox |P'|_{\delta} \gtrapprox \delta^{-s^{2}/(s + t)}  \Longrightarrow  \Delta \gtrapprox \delta^{1 - s/(s + t)} = \delta^{t/(s + t)}. \end{displaymath} 
Second, all the lines $\ell \in \mathcal{L}'$ with $R \subset [\ell]_{C\delta}$ are themselves contained in a $d_{\mathcal{A}(2)}$-ball of radius $\sim (\delta/\Delta)$. Consequently, now using the lower bound $|\mathcal{L}'|_{\delta} \gtrapprox \delta^{-t^{2}/(s + t)}$ and the Katz-Tao $(\delta,t)$-set condition of $\mathcal{L}' \subset \mathcal{L}$,
\begin{displaymath} \delta^{-t^{2}/(s + t)} \lessapprox |\mathcal{L}'|_{\delta} \lessapprox |\{\ell \in \mathcal{L}' : R \subset [\ell]_{C\delta}\}|_{\delta} \lessapprox \Delta^{-t}  \Longrightarrow  \Delta \lessapprox \delta^{t/(s + t)}. \end{displaymath}
Combining these inequalities gives \eqref{form39}. Therefore, combined with Proposition \ref{prop4}, Theorem \ref{main} provides the following geometric information: there exists a rectangle $R \subset \R^{2}$ of dimensions $\approx (1 \times \delta^{t/(s + t)})$ such that 
\begin{displaymath} |P \cap R|_{\delta} \gtrapprox \delta^{-s^{2}/(s + t)} \quad \text{and} \quad |\{\ell \in \mathcal{L} : R \subset [\ell]_{C\delta}\}|_{\delta} \gtrapprox \delta^{-t^{2}/(s + t)}. \end{displaymath}
The $\delta$-neighbourhoods of the sets $P \cap R$ and $\{\ell : R \subset [\ell]_{C\delta} \}$ are the "sheaves" in Figure \ref{fig3}. \end{remark}

\subsection{Related work and further problems} Theorem \ref{main} and Corollary \ref{cor3} give a characterisation of the extremal configurations in Fu and Ren's Theorem \ref{t:FuRen} when $s,t \in (0,1]$. It is a natural -- and difficult -- open problem to study the structure of extremal configurations in the original Szemer\'edi-Trotter incidence bound \eqref{st}. Of course, any answers (and methods) in this problem will be completely different from the one provided by Theorem \ref{main}: for example, if $P \times \mathcal{L}$ is a $(0,1)$-clique of points and lines, then $\min\{|\mathcal{L}|,|P|\} = 1$. For recent work on this discrete variant of the problem, see the papers of Solymosi \cite{MR2272219}, Sheffer-Silier, \cite{sheffer2022structural}, and Katz-Silier \cite{katz2023structure}.

In the $\delta$-discretised setting, we are not aware of previous structural results analogous to Theorem \ref{main}. On the other hand, Theorem \ref{main} is far from exhaustive. For example, it only covers the range $s,t \in (0,1]$ of the Fu-Ren incidence theorem. The reason is that the known sharpness examples in other ranges of $s,t$ have rather different structure than the "unions of cliques" shown in Figure \ref{fig3}. We are not even sure what to expect if $\max\{s,t\} > 1$, and certainly the required proof techniques would be different from ours.

Another further direction is to relax or change the non-concentration conditions we impose in Theorem \ref{main}. This will typically change the sharp upper bounds for $|\mathcal{I}(P,\mathcal{L})|_{\delta}$, and therefore the problem of characterising the extremal configurations. However, this is not always the case. For example, if $P \subset [0,1]^{2}$ is a Katz-Tao $(\delta,1)$-set, and $\mathcal{L} \subset \mathcal{A}(2)$ is \textbf{any} set of lines with $|\mathcal{L}|_{\delta} \leq \delta^{-1}$, then $|\mathcal{I}(P,\mathcal{L})|_{\delta} \lesssim \delta^{-3/2}$. This folklore result (see e.g. \cite[Proposition 2.13]{OS23}) matches Fu and Ren's bound in the case $s = t = 1$, and the Katz-Tao $(\delta,1)$-set condition on $\mathcal{L}$ is not needed. So, the following question makes sense: assume that $P \subset [0,1]^{2}$ is a Katz-Tao $(\delta,1)$-set, and $\mathcal{L} \subset \mathcal{A}(2)$ satisfies $|\mathcal{L}|_{\delta} \leq \delta^{-1}$. If $|\mathcal{I}(P,\mathcal{L})|_{\delta} \gtrapprox \delta^{-3/2}$, does the conclusion of Theorem \ref{main} (in the case $s = t = 1$) continue to hold? Our proof heavily relies on the Katz-Tao $(\delta,1)$-set properties of both $P$ and $\mathcal{L}$.

Finally, we refer the reader to further recent advances in the active area of estimating $\delta$-discretised incidences between points and lines in $\R^{d}$: \cite{MR4615465,MR4447307,GSW,MR2538607,2023arXiv231114481O,2023arXiv230409464P,2023arXiv230904097R,2023arXiv230808819R,ShmerkinWang22,MR3034673,2022arXiv221009581W}.

\subsection{Outline of the paper} In Section \ref{section 2} we gather preliminary results required to prove Theorem \ref{main}. The main technical result in that section is Proposition \ref{prop2} which may have some independent interest to experts.

The proof of Theorem \ref{main} occupies Section \ref{section 3}. There is a substantial difference between the complexity of the proofs when $s = t$ (harder) and $s \neq t$ (easier). For the case $s = t$, we need the non-trivial \cite[Theorem 5.61]{2023arXiv230110199O}, repeated as Theorem \ref{thm:Furstenberg-minimal-non-concentration}. This is a quantitative \emph{Furstenberg set estimate}, although not the sharp one from \cite{2023arXiv230808819R}. This auxiliary result is not required in the case $s \neq t$. It might have been possible to combine the cases $s = t$ and $s \neq t$, but we decided to separate them for clarity. Where the details are very similar, we give all of them in the harder case $s = t$, and a sketch when $s \neq t$. Regarding the cases $s \neq t$, we only give a (fairly) detailed argument for $s < t$, and then infer the cases $s > t$ by point-line duality (see Section \ref{s 3.3} for the details).

Finally, Section \ref{s4} contains the proofs of Corollary \ref{cor3} and Proposition \ref{prop4}.
		
\section{Preliminaries}\label{section 2}

\subsection{Notations and $(\delta,s)$-sets}
We adopt the standard notations $\lesssim$, $\gtrsim$, $\sim$. For example, $A\lesssim B$ means $A\leq CB$ for some constant $C>0$, while $A\lesssim_rB$ stands for $A\leq C(r)B$ for a positive function $C(r)$. We will denote $A \lessapprox_\delta B$, $A \gtrapprox_\delta B$, $A \approx_\delta B$ or $A\approx B$ to hide slowly growing functions of $\delta$ such as $\log(1/\delta)$ and $\delta^{-\epsilon}$. The precise meaning of the $\lessapprox$ notation will always be explained separately.

For $\delta \in 2^{-\N}$, dyadic $\delta$-cubes in $\R^{d}$ are denoted $\mathcal{D}_{\delta}(\R^{d})$. Elements of $\mathcal{D}_{\delta}(\R^{d})$ are typically denoted with letters $p,q$. For $P \subset \R^{d}$, we write $\mathcal{D}_{\delta}(P) := \{p \in \mathcal{D}_{\delta}(\R^{d}) : P \cap p \neq \emptyset\}$.

In addition to the Katz-Tao $(\delta,s)$-set condition (Definition \ref{def:katzTaoSet}), also the following slightly different non-concentration property will be needed in the paper:

\begin{definition}[$(\delta,s,C)$-set]\label{def:deltaSSet}
	For $\delta\in (0,1]$, $s\in[0,d]$ and $C>0$, a nonempty bounded set $P \subset \R^{d}$ is called a $(\delta,s,C)$-set if
	\begin{equation}\label{def 2}
		|P\cap B(x,r)|_\delta\leq Cr^s |P|_\delta,~~\forall x\in\mathbb R^d,~r\in[\delta,1].
	\end{equation}
	A family $\mathcal{P} \subset \mathcal{D}_{\delta}(\R^{d})$ is called a $(\delta,s,C)$-set if $\cup \mathcal{P} \subset \R^{d}$ is a $(\delta,s,C)$-set.
\end{definition}
Since both Definitions \ref{def:katzTaoSet} and \ref{def:deltaSSet} will be used in the paper, we will always be careful and explicit in either including the words "Katz-Tao", or omitting them.

\subsection{Point-line duality and dyadic tubes} 

\begin{definition}\label{line map} Let $D \colon \R^{2} \to \mathcal{A}(2)$ be the \emph{point-line duality map} sending $(a,b)$ to a corresponding line in $\R^2$, defined by
	\[D(a,b):= \ell_{a,b} := \{(x,y)\in\R^2: y=ax+b\} \in \mathcal{A}(2).\]
	 \end{definition}

The following useful lemma follows by chasing the definitions:

\begin{lemma}\label{lemma2} The map $D \colon ([-1,1] \times \R,|\cdot|) \to (\mathcal{A}(2),d_{\mathcal{A}(2)})$ is bi-Lipschitz. \end{lemma}

\begin{definition}[Dyadic $\delta$-tubes]\label{def:dyadicTube} Let $\delta \in 2^{-\N}$ and 
\begin{displaymath} Q = [a_{0},a_{0} + \delta) \times [b_{0},b_{0} + \delta)  \in \mathcal{D}_{\delta}([-1,1] \times \R). \end{displaymath}
The union of lines $T := \cup\{D(a,b) : (a,b) \in Q\} \subset \R^{2}$ is called a \emph{dyadic $\delta$-tube}. The \emph{slope of $T$} is defined to be $\sigma(T) := a_{0}$.  The family of dyadic $\delta$-tubes in $\R^{2}$ is denoted $\mathcal{T}^{\delta}$. 

If $\mathcal{L} \subset \mathcal{A}(2)$, we denote $\mathcal{T}^{\delta}(\mathcal{L})$ the family of dyadic tubes which contain at least one line from $\mathcal{L}$. Whenever $\mathcal{L} \subset D([-1,1) \times \R)$, the family $\mathcal{T}^{\delta}(\mathcal{L})$ is a cover of $\mathcal{L}$.  \end{definition}

By an abuse of notation and terminology, we sometimes view dyadic $\delta$-tubes as subsets of $\mathcal{A}(2)$. In fact, we already did so in the last sentence of Definition \ref{def:dyadicTube}. 

We introduce notation for "dyadic covers" of sets $\mathcal{I} \subset \R^{2} \times \mathcal{A}(2)$:
\begin{displaymath} \mathcal{D}_{\delta}(\mathcal{I}) := \{(p,T) \in \mathcal{D}_{\delta} \times \mathcal{T}^{\delta} : x \in p \text{ and } \ell \subset T \text{ for some } (x,\ell) \in \mathcal{I}\}. \end{displaymath}
To be accurate, the elements of $\mathcal{D}_{\delta}(\mathcal{I})$ only cover $\mathcal{I}$ when the $\mathcal{A}(2)$-component of $\mathcal{I}$ consists of non-vertical lines. We will only use this notation when $\mathcal{I} \subset \R^{2} \times D([-1,1) \times \R)$.

\begin{lemma}\label{lemma3} Let $P \subset \R^{2}$, and let $\mathcal{L} \subset D([-1,1) \times \R) \subset \mathcal{A}(2)$. Then,
\begin{equation}\label{form48} |\mathcal{I}(P,\mathcal{L})|_{\delta} \sim |\mathcal{D}_{\delta}(\mathcal{I}(P,\mathcal{L}))|. \end{equation}
\end{lemma}

\begin{proof} We start with the inequality "$\lesssim$". Let $(x_{1},\ell_{1}),\ldots,(x_{n},\ell_{n}) \in \mathcal{I}(P,\mathcal{L})$ be a maximal $(11\delta)$-separated set. For each $1 \leq j \leq n$, pick $(p_{j},T_{j}) \in \mathcal{D}_{\delta} \times \mathcal{T}^{\delta}$ with $x_{j} \in p_{j}$ and $\ell_{j} \subset T_{j}$. Then $(p_{j},T_{j}) \in \mathcal{D}_{\delta}(\mathcal{I}(P,\mathcal{L}))$, since $(x_{j},\ell_{j}) \in \mathcal{I}(P,\mathcal{L})$. Furthermore, the map $(x_{j},\ell_{j}) \mapsto (p_{j},T_{j})$ is injective, because if $(p,T) \in \mathcal{D}_{\delta} \times \mathcal{T}^{\delta}$ is fixed, then the set $\{(x,\ell) : x \in p \text{ and } \ell \subset T\}$ is contained in a $d_{\mathcal{A}(2)}$-ball of radius $5\delta$. Therefore, $|\mathcal{I}(P,\mathcal{L})|_{\delta} \sim n \leq  |\mathcal{D}_{\delta}(\mathcal{I}(P,\mathcal{L}))|$.

We then prove the inequality "$\gtrsim$". Write $\mathcal{D}_{\delta}(\mathcal{I}(P,\mathcal{L})) = \{(p_{1},T_{1}),\ldots,(p_{n},T_{n})\}$, and $T_{j} = \cup D(q_{j})$, where $q_{j} \in \mathcal{D}_{\delta}(\R^{2})$. We say that $(p_{i},T_{i})$ and $(p_{j},T_{j})$ are \emph{neighbours} if $\dist(p_{i},p_{j}) \leq C\delta$ and $\dist(q_{i},q_{j}) \leq C\delta$ for a suitable absolute constant $C \geq 1$. Pick any maximal neighbour-free subset $\mathcal{I} \subset \mathcal{D}_{\delta}(\mathcal{I}(P,\mathcal{L}))$. It is easy to check that $|\mathcal{I}| \sim_{C} n$. 

We claim that $|\mathcal{I}(P,\mathcal{L})|_{\delta} \geq |\mathcal{I}|$, which will complete the proof. To see this, pick $(p,T) \in \mathcal{I} \subset \mathcal{D}_{\delta}(\mathcal{I}(P,\mathcal{L}))$. Then, by definition there exist $x_{p} \in p$ and $\ell_{T} \subset T$ such that $(x_{p},\ell_{T}) \in \mathcal{I}(P,\mathcal{L})$, in particular $\ell_{T} \in \mathcal{L} \subset D([-1,1) \times \R)$.

 Now, it suffices to note that the pairs $(x_{p},\ell_{T}) \in \mathcal{I}(P,\mathcal{L})$ obtained this way are $\delta$-separated. If $(p,T),(p',T') \in \mathcal{I}$ are distinct, then either $\dist(p,p') \geq C\delta$ or $\dist(q,q') \geq C\delta$. In the former case $|x_{p} - x_{p'}| \geq 10\delta$. In the latter case $d_{\mathcal{A}(2)}(\ell_{T},\ell_{T'}) \gtrsim \dist(q,q') \geq C\delta$ by the bi-Lipschitz property of $D$. Therefore $d_{\mathcal{A}(2)}(\ell_{T},\ell_{T'}) \geq \delta$ is $C \geq 1$ is large enough. \end{proof}

\subsection{Incidence bounds}

The following result is a version of Fu and Ren's Theorem \ref{t:FuRen} where the dependence on the non-concentration constants has been quantified. It is also due to Fu-Ren, see \cite[Theorem 3.1 and Theorem 3.2]{fu2022incidence}.

\begin{thm}\label{combinatorial bound}
	Let $0\leq s, t\leq 1$ and $K_{P}, K_{\mathcal{L}} \geq 1$. Assume $P \subset [0,1]^{2}$ is a Katz-Tao $(\delta,s, K_{P})$-set and $\mathcal{L} \subset \mathcal{A}(2)$ is a Katz-Tao $(\delta,t, K_{\mathcal{L}})$-set. Then, 
	\begin{equation}\label{form56} |\mathcal{I}(P,\mathcal{L})|_{\delta}^{s+t} \lesssim_{\epsilon} \delta^{-st(1+\epsilon)} K_{P}^t K_{\mathcal{L}}^s|P|_{\delta}^s |\mathcal{L}|_{\delta}^t, \qquad \epsilon > 0. \end{equation}
\end{thm}

\begin{remark}\label{rem2} The original formulation \cite[Theorem 3.1 and Theorem 3.2]{fu2022incidence} of Theorem \ref{combinatorial bound} concerned incidences of the form $\overline{\mathcal{I}}(\mathcal{B},\mathcal{T}) = \{(B,T) \in \mathcal{B} \times \mathcal{T} : B \cap T \neq \emptyset\}$, where $\mathcal{B}$ is a family of $\delta$-discs and $\mathcal{T}$ is a family of $\delta$-tubes. Let us clarify why the original formulation implies Theorem \ref{combinatorial bound} as stated. (We give the full details to make sure that the original dependence on the constants $K_{P}$ and $K_{\mathcal{L}}$ can be maintained.)

Let $P,\mathcal{L}$ be as in Theorem \ref{combinatorial bound}, pick a maximal $(3\delta)$-separated set $(x_{1},\ell_{1}),\ldots,(x_{n},\ell_{n}) \in \mathcal{I}(P,\mathcal{L})$ in the $d$-metric of $\R^{2} \times \mathcal{A}(2)$. Thus, $|\mathcal{I}(P,\mathcal{L})|_{\delta} \sim n$, and $x_{j} \in \ell_{j}$ for all $1 \leq j \leq n$. Let $P' \subset \{x_{1},\ldots,x_{n}\}$ and $\mathcal{L}' \subset \{\ell_{1},\ldots,\ell_{n}\}$ be maximal $\delta$-separated sets, and consider the families of $(10\delta)$-balls and $(10\delta)$-tubes
\begin{displaymath} \mathcal{B} := \{B(x',10\delta) : x' \in P'\} \quad \text{and} \quad \mathcal{T} := \{[\ell]_{10\delta} : \ell \in \mathcal{L}'\}, \end{displaymath}
where $[\ell]_{r}$ is the $r$-neighbourhood of $\ell$. Then $\mathcal{B}$ is a Katz-Tao $(10\delta,s,O(K_{P}))$-set and $\mathcal{T}$ is a Katz-Tao $(10\delta,t,O(K_{\mathcal{L}}))$-set in the terminology of \cite{fu2022incidence}. Further, we claim that $n \leq |\overline{\mathcal{I}}(\mathcal{B},\mathcal{T})|$. To see this, fix $1 \leq j \leq n$. By the definitions of $P',\mathcal{L}'$, there exist $x' \in P'$, $\ell' \in \mathcal{L}'$ such that $|x_{j} - x'| \leq \delta$ and $|\ell_{j} - \ell'| \leq \delta$. Since $x_{j} \in \ell_{j}$,
\begin{displaymath} B(x',10\delta) \cap [\ell']_{10\delta} \neq \emptyset. \end{displaymath}
Moreover, the map $(x_{j},\ell_{j}) \mapsto (x',\ell')$ is injective: two pairs $(x_{i},\ell_{i})$ and $(x_{j},\ell_{j})$ corresponding to the same pair $(x',\ell')$ would satisfy $|x_{i} - x_{j}| \leq 2\delta$ and $|\ell_{i} - \ell_{j}| \leq 2\delta$, and therefore $d((x_{i},\ell_{i}),(x_{j},\ell_{j})) \leq 2\delta$, contrary to the $(3\delta)$-separation. This proves the inequality $n \leq |\overline{\mathcal{I}}(\mathcal{B},\mathcal{T})|$, and finally \eqref{form56} follows from the original formulation of \cite{fu2022incidence}. \end{remark}

Besides Theorem \ref{combinatorial bound}, a main tool in the proof of Theorem \ref{main} is \cite[Theorem 5.61]{2023arXiv230110199O} stated below as Theorem \ref{thm:Furstenberg-minimal-non-concentration}. To be accurate, the statement below is the "dual" version of \cite[Theorem 5.61]{2023arXiv230110199O} which is more convenient for our application. Another small difference is that Theorem \ref{thm:Furstenberg-minimal-non-concentration} is stated for ("ordinary") $\delta$-tubes, whereas \cite[Theorem 5.61]{2023arXiv230110199O} is formulated in terms of dyadic $\delta$-tubes. The introduction of dyadic $\delta$-tubes in \cite{2023arXiv230110199O} brings technical convenience in the proof, but the two versions are \emph{a posteriori} easily seen to be equivalent. In the statement, a $\delta$-tube is any rectangle of dimensions $\delta \times 1$, and two $\delta$-tubes $T,T'$ are called \emph{distinct} if $\mathrm{Leb}(T \cap T') \leq \tfrac{1}{2}\mathrm{Leb}(T)$.

\begin{thm}\label{thm:Furstenberg-minimal-non-concentration}
	Fix $\eta\in (0,1]$, $t\in (0,2)$, $u \in (0,\min\{t,2 - t\}]$, and $0 \leq \alpha < \eta u/4$. There exists $\epsilon=\epsilon(\eta,t,u) > 0$ and $\delta_{0} = \delta_{0}(\alpha,\eta,t,u) > 0$ such that the following holds for all $\delta \in (0,\delta_{0}]$.
	
	Let $\mathcal{T}$ be a family of distinct $\delta$-tubes with $|\mathcal{T}| = \delta^{-t}$ and satisfying the following non-concentration condition at the single scale $\rho := \delta|\mathcal{T}|^{1/2}$:
	\begin{equation}\label{form10} |\{T \in \mathcal{T} : T \subset \mathbf{T}\}| \leq \delta^{u} |\mathcal{T}|, \end{equation}
	where $\mathbf{T} \subset \R^{2}$ is an arbitrary $(\rho \times 2)$-rectangle. Let $N \geq 1$. For every $T \in \mathcal{T}$, assume that there exists a $(\delta,\eta,\delta^{-\epsilon})$-set $\mathcal{P}_{T} \subset \mathcal{D}_{\delta}([0,1)^{2})$ satisfying $|\mathcal{P}_{T}| \geq N$, and with the property that every square in $\mathcal{P}_{T}$ intersects $T$. Then,
	\begin{displaymath} 
		\left| \bigcup_{T \in \mathcal{T}} \mathcal{P}_{T} \right| \geq N \cdot |\mathcal{T}|^{1/2} \cdot \delta^{-\alpha}. 
	\end{displaymath}
\end{thm}

While Theorem \ref{thm:Furstenberg-minimal-non-concentration} is an improvement (enabled by \eqref{form10}) over the classical "$2$-ends" incidence bound, we will also employ the classical bound, recorded below:
\begin{proposition}\label{prop3} Let $\mathcal{T}$ be a family of of dyadic $\delta$-tubes or distinct (ordinary) $\delta$-tubes, $M \geq 1$ and $r > 0$. For every $T \in \mathcal{T}$, assume that there exists a set $\mathcal{P}_{T} \subset \mathcal{D}_{\delta}([0,1)^{2})$ with $|\mathcal{P}_{T}| = M$, with the property that every square in $\mathcal{P}_{T}$ intersects $T$, and $\mathcal{P}_{T}$ satisfies the following \emph{$2$-ends condition}:
	\begin{equation}\label{form11} |\mathcal{P}_{T} \cap B(x,r)| \leq \tfrac{1}{3}M, \qquad x \in \R^{2}. \end{equation}
	Then,
	\begin{equation}\label{form12} 
		\left| \bigcup_{T \in \mathcal{T}} \mathcal{P}_{T} \right| \gtrsim |\mathcal{T}|^{1/2} \cdot M \cdot r^{1/2}. 
	\end{equation}
\end{proposition} 

\begin{proof} 
	According to \eqref{form11}, we may for every $T \in \mathcal{T}$ find two subsets $\mathcal{P}^{1}_{T},\mathcal{P}^{2}_{T} \subset \mathcal{P}_{T}$ such that $|\mathcal{P}_{T}^{j}| \sim M$ for $j \in \{1,2\}$, and $\dist(p,q) \gtrsim r$ for all $(p,q) \in \mathcal{P}^{1}_{T} \times \mathcal{P}^{2}_{T}$. Consequently,
	\begin{displaymath} 
		\sum_{T \in \mathcal{T}} |\mathcal{P}_{T}^{1} \times \mathcal{P}_{T}^{2}| \gtrsim |\mathcal{T}| \cdot M^{2}. 
	\end{displaymath} 
	On the other hand, denoting by "$\mathcal{P}$" the set appearing in \eqref{form12}, we have
	\begin{displaymath} 
		\sum_{T \in \mathcal{T}} |\mathcal{P}_{T}^{1} \times \mathcal{P}_{T}^{2}| \leq \mathop{\sum_{(\mathbf{p},\mathbf{q}) \in \mathcal{P}^{2}}}_{\dist(\mathbf{p},\mathbf{q}) \gtrsim r} |\{T \in \mathcal{T} : \mathbf{p} \cap T \neq \emptyset \neq \mathbf{q} \cap T\}| \lesssim |\mathcal{P}|^{2}/r. 
	\end{displaymath}
	Combining these estimates gives \eqref{form12}. \end{proof}

\subsection{Uniform sets} The items in this section are repeated from \cite[Section 2.3]{2023arXiv230110199O}.
	
\begin{definition}\label{def:uniformity}
Let $n \geq 1$, and let
\begin{displaymath} \delta = \Delta_{n} < \Delta_{n - 1} < \ldots < \Delta_{1} \leq \Delta_{0} = 1 \end{displaymath}
be a sequence of dyadic scales.  We say that a set $P\subset [0,1)^2$ is \emph{$\{\Delta_j\}_{j=1}^n$-uniform} if there is a sequence $\{N_j\}_{j=1}^n$ such that $N_{j} \in 2^{\N}$ and $|P\cap Q|_{\Delta_{j}} = N_j$ for all $j\in \{1,\ldots,n\}$ and all $Q\in\mathcal{D}_{\Delta_{j - 1}}(P)$.
\end{definition}

The following simple but key lemma asserts that one can always find ``dense uniform subsets''. See e.g. \cite[Lemma 3.6]{Sh} for the short proof.
\begin{lemma} \label{l:uniformization}
Let $P\subset [0,1)^{d}$, $m,H \in \N$, and $\delta := 2^{-mH}$. Let also $\Delta_{j} := 2^{-jH}$ for $0 \leq j \leq m$, so in particular $\delta = \Delta_{m}$. Then, there is a $\{\Delta_j\}_{j=1}^{m}$-uniform set $P'\subset P$ such that
\begin{displaymath}
|P'|_\delta \ge  \left(2H \right)^{-m} |P|_\delta.
\end{displaymath}
In particular, if $\epsilon > 0$ and $H^{-1}\log (2H) \leq \epsilon$, then $|P'|_{\delta} \geq \delta^{\epsilon}|P|_{\delta}$.
\end{lemma}
	
The lemma has the following superficially stronger corollary, which we will also need. The details can be found in \cite[Corollary 7.9]{2023arXiv230110199O}.
	
\begin{cor}\label{cor1} For every $\epsilon > 0$, there exists $H_{0} = H_{0}(\epsilon) \geq 1$ such that the following holds for all $\delta = 2^{-mH}$ with $m \geq 1$ and $H \geq H_{0}$. Let $\mathcal{P} \subset \mathcal{D}_{\delta}$. Then, there exist disjoint $\{2^{-jH}\}_{j = 1}^{m}$-uniform subsets $\mathcal{P}_{1},\ldots,\mathcal{P}_{N} \subset \mathcal{P}$ with the properties
\begin{itemize}
\item $|\mathcal{P}_{j}| \geq \delta^{2\epsilon}|\mathcal{P}|$ for all $1 \leq j \leq N$,
\item $|\mathcal{P} \, \setminus \, (\mathcal{P}_{1} \cup \ldots \cup \mathcal{P}_{N})| \leq \delta^{\epsilon}|\mathcal{P}|$.
\end{itemize}
\end{cor}
	
\begin{definition}[Branching function]\label{branchingFunction} Let $H \in \N$, and let $P \subset [0,1)^{d}$ be a $\{\Delta_{j}\}_{j = 1}^{m}$-uniform set, with $\Delta_{j} := 2^{-jH}$, and let $\{N_{j}\}_{j = 1}^{m} \subset \{1,\ldots,2^{dH}\}^{m}$ be the associated sequence. We define the \emph{branching function} $\beta \colon [0,m] \to [0,dm]$ by setting $\beta(0) = 0$, and
		\begin{displaymath} \beta(j) := \frac{\log |P|_{2^{-jH}}}{H} = \frac{1}{H} \sum_{i = 1}^{j} \log N_{i}, \qquad i \in \{1,\ldots,m\}, \end{displaymath}
		and then interpolating linearly. \end{definition}
		
	\begin{definition}[$\epsilon$-linear and superlinear functions] \label{def:eps-linear}
		Given a function $f:[a,b]\to\R$ and numbers $\epsilon,\sigma \geq 0$, we say that $(f,a,b)$ is \emph{$(\sigma,\epsilon)$-superlinear} if
		\[
		f(x) \ge f(a)+\sigma(x-a)-\epsilon(b-a),\qquad x\in [a,b].
		\]
		If $\epsilon = 0$, we simply say that $(f,a,b)$ is $\sigma$-superlinear.
	\end{definition}

The following lemma is \cite[Lemma 8.3]{OS23}, but we give the proof to record the dependence on the constant $\Delta$ more explicitly.

\begin{lemma}\label{rescaled set}
Let $P\subset [0,1)^{d}$ be \emph{$\{\Delta^j\}_{j=1}^m$}-uniform with branching function $\beta$ and let $\delta=\Delta^m$. If $(\beta, a, b)$ is $s$-superlinear for any integers $0\leq a< b\leq m$ and $s>0$, then for any $Q\in\mathcal{D}_{\Delta^a}(P)$, the rescaled set $S_Q(P\cap Q)$ is a $(\Delta^{b-a}, s, C\Delta^{-s})$-set for some $C=C(d)>0$.
\end{lemma}
\begin{proof}
By \cite[Lemma 2.25]{2023arXiv230110199O}, $S_Q(P)$ is \emph{$\{\Delta^j\}_{j=1}^{m-a}$}-uniform for any $Q\in\mathcal{D}_{\Delta^a}(P)$, and the corresponding branching function $\beta_Q$ satisfies
\[\beta_Q(x)=\beta(x+a) - \beta(a),~~x\in[0,m-a],\]
and $(\beta_Q, 0, b-a)$ is $s$-superlinear. For any $q\in\mathcal{D}_{\Delta^i}(P)$, $0\leq i\leq b-a$, we have
\begin{equation*}
\begin{split}
|S_Q(P)\cap q|_{\Delta^{b-a}} & =N_{i+1}N_{i+2}\cdots N_{m-a}=\frac{|S_Q(P)|_{\Delta^{b-a}}}{N_1N_2\cdots N_i}\\
& = |S_Q(P)|_{\Delta^{b-a}}2^{-\log(\Delta^{-1})\beta_Q(i)}\\
& \leq |S_Q(P)|_{\Delta^{b-a}}2^{-\log(\Delta^{-1})is}=\Delta^{is}|S_Q(P)|_{\Delta^{b-a}}.
\end{split}
\end{equation*}
In general, for any $r\in[\Delta^{b-a},1]$, there exists $i$ such that $r\in[\Delta^{i+1},\Delta^i)$ and thus $\Delta^i\leq \Delta^{-1}r$. For any  $q\in\mathcal{D}_r(S_Q(P))$, choose any $q_1\in\mathcal{D}_{\Delta^i}(P)$, we simply get
\begin{equation*}
\begin{split}
|S_Q(P)\cap q|_{\Delta^{b-a}}\leq|S_Q(P)\cap q_1|_{\Delta^{b-a}}\leq \Delta^{is}|S_Q(P)|_{\Delta^{b-a}}\leq \Delta^{-s} r^s |S_Q(P)|_{\Delta^{b-a}},
\end{split}
\end{equation*}
as required. By uniformity of $P$, we conclude that $S_Q(P\cap Q)$ is a $(\Delta^{b-a}, s, C\Delta^{-s})$-set.
\end{proof}

The following lemma is \cite[Lemma 5.19]{SW21}:
\begin{lemma}\label{lemma1} Assume that $(f,a,b)$ is $\sigma_{1}$-superlinear and $(f,b,c)$ is $\sigma_{2}$-superlinear with $\sigma_{1} \geq \sigma_{2} \geq 0$. Then, $(f,a,c)$ is $\sigma$-superlinear with
\begin{displaymath} \sigma = \tfrac{b - a}{c - a} \cdot \sigma_{1} + \tfrac{c - b}{c - a} \cdot \sigma_{2} \in [\sigma_{2},\sigma_{1}]. \end{displaymath}
\end{lemma}
	
The following result is a slight variant of \cite[Lemma 2.1]{2023arXiv230110199O}. In fact, our variant is easier to prove, so we give the full details below the statement:
	
\begin{lemma}\label{l:combinatorial-weak}
Let $f \colon [0,m] \to \R$ be a non-decreasing piecewise affine $d$-Lipschitz function with $f(0) = 0$. Then, there exist sequences
\begin{align}
0&=a_0 < a_1 <\cdots < a_{n} = m, \notag\\
\label{form23} 0&\le \sigma_0<\sigma_1<\cdots <\sigma_{n-1} \le d,
\end{align}
such that:
\begin{enumerate}
\item $(f,a_{j},a_{j + 1})$ is $\sigma_{j}$-superlinear. 
\item $\sum_{j = 0}^{J - 1} \sigma_{j}(a_{j + 1} - a_{j}) = f(a_{J})$ for all $1 \leq J \leq n$.
\end{enumerate}
\end{lemma}
	
\begin{remark}\label{rem1} It will be useful to note that (2) is equivalent to 
\begin{displaymath} f(a_{j + 1}) - f(a_{j}) = \sigma_{j}(a_{j + 1} - a_{j}), \qquad 0 \leq j \leq n - 1. \end{displaymath}
\end{remark}
	
\begin{proof}  The proof in a picture is shown in Figure \ref{fig4}. Since $f$ is piecewise affine, there exists an initial partition of $[0,m]$ into intervals $[A_0,A_1],\ldots,[A_{N - 1}, A_N]$ such that $f$ is affine on each $[A_j,A_{j + 1}]$. In particular, $(f,A_j,A_{j + 1})$ is $\Sigma_j$-(super)linear with $\Sigma_j = [f(A_{j + 1}) - f(A_j)]/(A_{j + 1} - A_j) \leq d$. Now, the points $\{A_j\}$ and the slopes $\{\Sigma_j\}$ would otherwise satisfy all the requirements of the lemma, except that \eqref{form23} may fail: the slopes $\Sigma_j$ may not be in increasing order.
\begin{figure}[h!]
\begin{center}
\begin{overpic}[scale = 0.6]{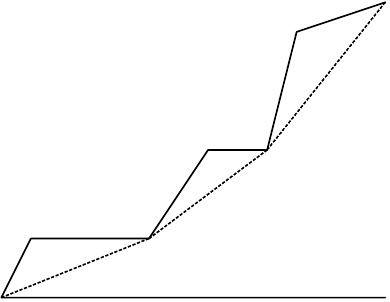}
\put(21,21){$f$}
\put(20,4){\small{$\sigma_{0}$}}
\put(54,24){\small{$\sigma_{1}$}}
\put(85,55){\small{$\sigma_{2}$}}
\end{overpic}
\caption{The function $f$, and the function determined by the slopes $\sigma_{j}$.}\label{fig4}
\end{center}
\end{figure}
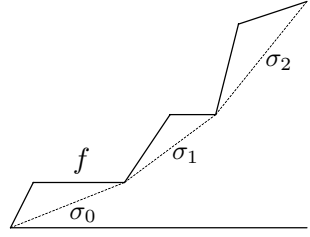		
To fix this, we proceed by induction: the plan is to merge consecutive intervals appropriately until \eqref{form23} is satisfied. The price to pay is that $f$ will lose its linearity on the new intervals, but the superlinearity property (1) will be retained. Let us assume that we have already defined a partition $[b_{0},b_{1}],\ldots,[b_{N_{t} - 1},b_{N_{t}}]$ with $N_{t} \leq N$, where $N_{0} = 0$. Let us also assume that we have associated to each interval $[b_{j},b_{j + 1}]$ a slope $\sigma_{j,t}$ such that properties (1)-(2) are valid (notationally adjusted, of course). 
		
Let us assume that \eqref{form23} still fails. Thus, there exists a pair of consecutive slopes $\sigma_{j,t},\sigma_{j + 1,t}$ with $\sigma_{j + 1,t} \leq \sigma_{j,t}$. We merge the intervals $[b_{j},b_{j + 1}]$ and $[b_{j + 1},b_{j + 2}]$ into a new interval $[b,b']$, and to this interval we associate the slope $\sigma$ given by Lemma \ref{lemma1}, namely
\begin{displaymath} \sigma = \tfrac{b_{j + 1} - b_{j}}{b_{j + 2} - b_{j}} \cdot \sigma_{j} + \tfrac{b_{j + 2} - b_{j + 1}}{b_{j + 2} - b_{j}} \cdot \sigma_{j + 1} \in [\sigma_{j+1},\sigma_{j}] \subset [0,d]. \end{displaymath}
The $\sigma$-superlinearity of $f$ on $[b,b']$ follows from Lemma \ref{lemma1}. The next partition is formed -- in the obvious way -- by the previous intervals, except that $[b_{j},b_{j + 1}],[b_{j + 1},b_{j + 2}]$ are replaced by the single interval $[b,b']$. In particular, $N_{t + 1} = N_{t} - 1$. This shows that the construction must terminate in no more than $N$ steps.
		
The only thing to check is that property (2) is retained by the generation $(t + 1)$ partition, and the new slopes. In fact, with Remark \ref{rem1} in mind, (2) reduces to verifying that $\sigma(b' - b) = f(b') - f(b)$. But $b' = b_{j + 2}$ and $b = b_{j}$, so by the definition of $\sigma$, and the inductive hypothesis, we have
\begin{align*} \sigma(b' - b) & = \sigma_{j}(b_{j + 1} - b_{j}) + \sigma_{j + 1}(b_{j + 2} - b_{j + 1})\\
& = [f(b_{j + 1}) - f(b_{j})] + [f(b_{j + 2}) - f(b_{j + 1})] = f(b') - f(b). \end{align*} 
This shows that property (2) is retained, and the construction may proceed.
		
As mentioned already, after no more than $N$ steps the construction must terminate: at this point the remaining slopes are in (strictly) increasing order, and the partitions and slopes at that step are the ones we were after. \end{proof}
	
\subsection{Finding non-concentrated subsets} One further key tool in the proof of Theorem \ref{main} is the next proposition, which allows us to find reasonably large reasonably non-concentrated subsets within arbitrary families of $\delta$-cubes:
	
\begin{proposition}\label{prop2} For every $d \in \N$, $C \geq 1$, and $t > 0$ there exists $\eta_{0} = \eta_{0}(C,t) \in (0,d/C)$ and $\delta_{0} = \delta_{0}(C,t) > 0$ such that the following holds for all $\delta \in 2^{-\N} \cap (0,\delta_{0}]$. Let $\mathcal{P} \subset \mathcal{D}_{\delta}$ be a family with $|\mathcal{P}| = \delta^{-t}$. Then, there exists a scale $\Delta \in 2^{-\N} \cap [\delta,1]$, a number $\eta \in [\eta_{0},d/C]$, a cube $Q \in \mathcal{D}_{\Delta}(\mathcal{P})$, and a subset $\mathcal{P}_{Q} \subset \mathcal{P} \cap Q$ with the following properties:
\begin{enumerate}
\item $|\mathcal{P}_{Q}| \geq \delta^{\eta}|\mathcal{P}|$, and
\item $S_{Q}(\mathcal{P}_{Q})$ is a $((\delta/\Delta),C\eta,O_{C,t}(1))$-set.
\end{enumerate}
Here $S_{Q}$ is the affine homothety mapping $Q$ to $[0,1)^{d}$.
\end{proposition}

\begin{remark} The main point of the proposition is the distinction between passing to a subset $\mathcal{P}_{Q}$ of cardinality $\geq \delta^{\eta}|\mathcal{P}|$, and gaining the $(C\eta)$-dimensional non-concentration condition for $S_{Q}(\mathcal{P}_{Q})$ -- for any prescribed $C \geq 1$. We also note that the non-concentration condition in (2) refers to Definition \ref{def:deltaSSet}, and not the Katz-Tao condition. \end{remark}
	
\begin{proof}[Proof of Proposition \ref{prop2}] Fix $C \geq 1$, $t > 0$ as in the statement, and let $\eta_{0} > 0$ be so small that
\begin{equation}\label{form17} 
\eta_{0} \cdot \exp\left( \tfrac{1}{2}C^{2} \right) < t. 
\end{equation} 
Applying initially Lemma \ref{l:uniformization} with "$\eta_{0}$" in place of "$\epsilon$", we may find a $\{2^{-jH}\}_{j = 1}^{m}$-uniform subset $\mathcal{P}' \subset \mathcal{P}$, where $H^{-1}\log(2H) \leq \eta_{0}$, and $|\mathcal{P}'| \geq \delta^{\eta_{0}/2}|\mathcal{P}|$. After this initial step, our efforts will be directed towards finding the numbers $\eta \in [\eta_{0},d/C]$ and the subset $\mathcal{P}_{Q}$ inside $\mathcal{P}'$ instead of $\mathcal{P}$, satisfying $|\mathcal{P}_{Q}| \geq \delta^{\eta/2}|\mathcal{P}'|$. Therefore finally $|\mathcal{P}_{Q}| \geq \delta^{\eta/2 + \eta_{0}/2}|\mathcal{P}| \geq \delta^{\eta}|\mathcal{P}|$. To simplify notation, we will continue denoting $\mathcal{P}'$ by $\mathcal{P}$ -- or in other words we assume without loss of generality that $\mathcal{P}$ is $\{2^{-jH}\}_{j = 1}^{m}$-uniform to start with.
		
Let $\beta \colon [0,m] \to [0,dm]$ be the branching function of $\mathcal{P}$. In particular, $\beta$ is an increasing piecewise affine $d$-Lipschitz function with $\beta(0) = 0$ and $2^{H\beta(m)} = |\mathcal{P}| = \delta^{-t} = 2^{Hm t}$, or in other words $\beta(m) = mt$. We apply Lemma \ref{l:combinatorial-weak} to the function $\beta$. This produces sequences
\begin{displaymath} 0 = a_{0} < a_{1} < \ldots < a_{n} = m \end{displaymath}
and
\begin{equation}\label{form9} 0 \leq \sigma_{0} < \sigma_{1} < \ldots < \sigma_{n - 1} \leq d \end{equation}
such that
\begin{itemize}
\item[(a)] $(\beta,a_{j},a_{j + 1})$ is $\sigma_{j}$-superlinear, and
\item[(b)] $\sum_{j = 0}^{J - 1} (a_{j + 1} - a_{j})\sigma_{j} = \beta(a_{J})$ for all $1 \leq J \leq n$.
\end{itemize}
		
Let $f \colon [0,m] \to [0,dm]$ be the function determined by the intervals $[a_{j},a_{j + 1}]$, the slopes $\sigma_{j}$, and the initial condition $f(0) = \beta(0) = 0$ (thus we require that $f'|_{(a_{j},a_{j + 1})} \equiv \sigma_{j}$). It follows from properties (a)-(b) that $\beta \geq f$, and $\beta(a_{j}) = f(a_{j})$ for all $0 \leq j \leq n$. Moreover, the function $f$ is piecewise affine and convex by \eqref{form9}.
		
Let us consider the renormalised version of $f$ defined by 
\begin{displaymath} g(x) := \tfrac{1}{m}f(mx), \qquad x \in [0,1]. \end{displaymath}
Then $g \colon [0,1] \to [0,t]$ is piecewise affine, $d$-Lipschitz, convex, and satisfies $g(0) = 0$ and $g(1) = t > \eta_{0}$. For $0 \leq j \leq n - 1$, we readily see $g_+'(a_{j}/m) \equiv \sigma_{j}$, namely the value of $g'$ on the interval $(a_{j}/m,a_{j + 1}/m)$. Here $g_+'$ is the right hand derivative.
		
We claim that there exists a point $x_{0} \in [0,1]$, in fact $x_{0} = a_{j}/m$ for some $0 \leq j \leq n - 1$, such that
\begin{equation}\label{form20} g_+'(x_{0}) \geq C\max\{\eta_{0},g(x_{0})\}. \end{equation}
To this end, let $a := \min\{x \in [0,1] : g(x) = \eta_{0}\}$ (this is well-defined, since $g(0) = 0$ and $g(1) = t > \eta_{0}$). We claim that there exists a point $x_{1} \in [a,1)$ such that $g'(x_{1}) \geq Cg(x_{1})$. Otherwise, the converse inequality $g'(x) \leq Cg(x)$ is true for a.e. $x \in [a,1)$. Since $g$ is absolutely continuous, we may deduce from Gr\"onwall's inequality that
\begin{equation}\label{form18} t = g(1) \leq g(a) \cdot \exp \left(\int_{a}^{1} C \, ds \right) \leq \eta_{0} \cdot \exp\left(\tfrac{1}{2}C^{2} \right) \stackrel{\eqref{form17}}{<} t, \end{equation}
leading to a contradiction. So, $x_{1} \in [a,1)$ satisfying \eqref{form20} exists. 
		
If $x_{1} = a_{j}/m$ for some $0 \leq j \leq n - 1$, we set $x_{0} := x_{1}$. Otherwise, we may replace $x_{1}$ by a point of the form $x_{0} = a_{j}/m$ still satisfying \eqref{form20}. Indeed, since $g$ is increasing and has constant derivative on every interval $(a_{j}/m,a_{j + 1}/m)$, we simply select $x_{0} = a_{j}/m$ to be the left end-point of the interval $(a_{j}/m,a_{j + 1}/m)$ where $x_{1} \in [a,1)$ happens to lie. Then $g_+'(x_{0}) = g'(x_{1}) \geq Cg(x_{1}) \geq Cg(x_{0})$, and also $g_+'(x_{0}) = g'(x_{1}) \geq Cg(x_{1}) \geq C\eta_{0}$. Thus \eqref{form20} is established.
		
We then set $\eta := \max\{g(x_{0}),\eta_{0}\}$ and $y_{0} := mx_{0} \in \{a_{0},\ldots,a_{n - 1}\}$. We note that $\eta \leq g_+'(x_{0})/C \leq d/C$, since $g$ is $d$-Lipschitz. Moreover, since $g$ is convex, we then have
\begin{displaymath} f'(y) = g'(y/m) \geq g_+'(x_{0}) \geq C\eta, \qquad y \in [y_{0},m], \end{displaymath} 
whenever the derivative exists. In particular, using that $\beta \geq f$, and $\beta(y_{0}) = f(y_{0})$ (this is why we needed to ensure that $x_{0} = a_{j}/m$) we find
\begin{displaymath} \beta(y) \geq f(y) \geq f(y_{0}) + C\eta(y - y_{0}) = \beta(y_{0}) + C\eta(y - y_{0}), \qquad y \in [y_{0},m]. \end{displaymath}
This means $(\beta, y_0, m)$ is $C\eta$-superlinear. By Lemma \ref{rescaled set}, for any $Q \in \mathcal{D}_{\Delta}(\mathcal{P})$ with $\Delta:= 2^{-y_{0}H}$, the rescaled set $S_{Q}(\mathcal{P} \cap Q)$ is a $((\delta/\Delta),C\eta,O((2^{H})^{C\eta}))$-set. Here $O((2^{H})^{C\eta}) = O_{C,t}(1)$ is a constant depending on $\eta,\eta_{0}$, hence $C,t$. We set $\mathcal{P}_{Q} := \mathcal{P} \cap Q$ for any fixed $Q \in \mathcal{D}_{\Delta}(\mathcal{P})$. Then part (2) of the proposition has been proven.
		
It remains to show that $|\mathcal{P}_{Q}| \geq \delta^{\eta}|\mathcal{P}|$. By the $\{2^{-jH}\}_{j = 1}^{m}$-uniformity of $\mathcal{P}$, we have
\begin{displaymath} |\mathcal{P}_{Q}| = \frac{|\mathcal{P}|}{|\mathcal{P}|_{\Delta}}. \end{displaymath}
On the other hand, using that $f(y_{0}) = \beta(y_{0})$ and $g(x_{0}) \leq \eta$, we have
\begin{displaymath} 
|\mathcal{P}|_{\Delta} = |\mathcal{P}|_{2^{-y_{0}H}} = 2^{H\beta(y_{0})} = 2^{Hf(y_{0})} = 2^{Hm g(x_{0})} \leq 2^{Hm \eta} = \delta^{-\eta}. 
\end{displaymath} 
Combining these facts shows that $|\mathcal{P}_{Q}| \geq \delta^{\eta}|\mathcal{P}|$, as desired.  \end{proof}

\section{Proof of Theorem \ref{main}} \label{section 3}
\subsection{Case $s = t$}\label{subsection 3.1} We restate Theorem \ref{main} in the case $s = t$:
\begin{thm}\label{main2}
	For every $s, u\in(0,1]$, there exist $\delta_0=\delta_0(s,u)>0$ and $\epsilon=\epsilon(s,u)>0$ such that the following holds for $\delta\in (0, \delta_0]$. Let $P \subset [0,1]^{2}$ and $\mathcal{L} \subset \mathcal{A}(2)$ be Katz-Tao $(\delta,s,\delta^{-\epsilon})$-sets. If
	\begin{equation}\label{incidence2}
		|\mathcal{I}(P,\mathcal{L})|_{\delta} \geq \delta^{\epsilon-3s/2},
	\end{equation}
	then there exists a $(\delta,\delta^{u})$-clique $P' \times \mathcal{L}' \subset P \times \mathcal{L}$ with 
	\begin{displaymath} |P'|_{\delta} \geq \delta^{u - s/2} \quad \text{and} \quad |\mathcal{L}'|_{\delta} \geq \delta^{u -s/2}. \end{displaymath}
\end{thm}

From now on, we fix the parameters $s, u \in(0,1]$, as in the statement of Theorem \ref{main2}. The parameter $\epsilon > 0$ will be determined in the proof, see \eqref{defEpsilon}. We record that in the case $s = t$ the incidence inequality of Fu-Ren in Theorem \ref{combinatorial bound} simplifies to
\begin{equation}\label{form24} |\mathcal{I}(P,\mathcal{L})|_{\delta} \leq C_{\epsilon} \delta^{-\epsilon}\sqrt{\delta^{-s}K_{\mathcal{P}}K_{\mathcal{T}}|P|_{\delta}|\mathcal{L}|_{\delta}}, \qquad \epsilon > 0. \end{equation}
We will use $A \lessapprox B$ to signify that there exists a constant $C > 0$, depending only on $s, u$, such that $A \leq C\delta^{-C\epsilon}B$. The two-sided inequality $A \lessapprox B \lessapprox A$ is abbreviated to $A \approx B$.

\begin{proof}[Proof of Theorem \ref{main2}] After an initial reduction performed right away, the proof will be divided into \textbf{Steps 1-3.} The initial reduction is this: we may assume that the lines $\ell = \ell_{a,b} \in \mathcal{L}$ have slopes $a \in [-1,1]$. In fact, there always exists a subset $\mathcal{L}' \subset \mathcal{L}$ such that (a) every pair of lines from $\mathcal{L}'$ forms an angle $\leq \tfrac{1}{10}$ and (b) $|\mathcal{I}(P,\mathcal{L}')|_{\delta} \sim |\mathcal{I}(P,\mathcal{L}')|_{\delta}$. We may then rotate both $P$ and $\mathcal{L}'$ such that the $[-1,1]$-slope condition is satisfied, and afterwards we proceed to find a $(\delta,\delta^{u})$-clique inside $P \times \mathcal{L}'$.

The $[-1,1]$-slope hypothesis is equivalent to $\mathcal{L} \subset D([-1,1] \times \R)$, so Lemma \ref{lemma3} will now allow us to express the $\delta$-covering number $|\mathcal{I}(P',\mathcal{L}')|_{\delta}$, for $P' \subset P$ and $\mathcal{L}' \subset \mathcal{L}$, in the convenient dyadic form $|\mathcal{D}_{\delta}(\mathcal{I}(P',\mathcal{L}'))|$.

\subsection*{Step 1. Reduction to the case where $\mathcal{D}_{\delta}(P)$ is uniform.} We first apply Corollary \ref{cor1} with parameter $\sqrt{\epsilon}$ to the family $\mathcal{P} := \mathcal{D}_{\delta}(P)$. This produces a constant $H \sim_{\epsilon} 1$ and a list of disjoint $\{2^{-jH}\}_{j = 1}^{m}$-uniform subsets $\mathcal{P}_{1},\ldots,\mathcal{P}_{N_0} \subset \mathcal{P}$ such that $|\mathcal{P}_{j}| \geq \delta^{2\sqrt{\epsilon}}|\mathcal{P}|$, and the "remainder set"
\begin{displaymath} \mathcal{R} := \mathcal{P} \, \setminus \, (\mathcal{P}_{1} \cup \ldots \cup \mathcal{P}_{N_0}) \end{displaymath}
satisfies $|\mathcal{R}| \leq \delta^{\sqrt{\epsilon}}|\mathcal{P}| \leq \delta^{\sqrt{\epsilon} - \epsilon - s}$. In particular, we can get $N_0 \leq \delta^{-2\sqrt{\epsilon}}$. Since $\mathcal{R}$ is also a Katz-Tao $(\delta,s, \delta^{-\epsilon})$-set, using \eqref{form24}, and the upper bound of $|\mathcal{R}|$ gives
\begin{displaymath}
|\mathcal{I}(\mathcal{R},\mathcal{L})|_{\delta} \lesssim_{\epsilon} \delta^{-\epsilon}\sqrt{\delta^{-s - 2\epsilon}|\mathcal{R}||\mathcal{L}|_{\delta}} \leq \delta^{\sqrt{\epsilon}/2-3\epsilon}\delta^{-3s/2}\leq \tfrac{1}{2}|\mathcal{I}(P,\mathcal{L})|_{\delta},
\end{displaymath}
provided that $\epsilon<1/64$ and $\delta > 0$ is small enough. Since $P$ is contained in the union of $\mathcal{R}$ and $\{\mathcal{P}_{j}\}$, there exists at least one index $j \in \{1,\ldots,N_0\}$ such that
\begin{displaymath} |\mathcal{I}(P \cap \mathcal{P}_{j},\mathcal{L})|_{\delta} \geq \tfrac{1}{2}N_0^{-1}|\mathcal{I}(P,\mathcal{L})|_{\delta} \gtrsim \delta^{3\sqrt{\epsilon} - 3s/2}. \end{displaymath} 
Here $\mathcal{P}_{j} = \mathcal{D}_{\delta}(P \cap \mathcal{P}_{j})$ is $\{2^{-jH}\}_{j = 1}^{m}$-uniform. Thus, at the cost of replacing "$\epsilon$" by "$3\sqrt{\epsilon}$" in the hypothesis \eqref{incidence2}, we can assume that $\mathcal{P}$ is $\{2^{-jH}\}_{j = 1}^{m}$-uniform for some $H \sim_{\epsilon} 1$.
	
Write $\mathcal{T} := \mathcal{T}^{\delta}(\mathcal{L})$. We claim that $|\mathcal{P}| \geq \delta^{6\epsilon -s}$ and $ |\mathcal{T}| \geq \delta^{6\epsilon -s}$, thus $|\mathcal{P}| \approx \delta^{-s} \approx |\mathcal{T}|$. Indeed, otherwise one can check that \eqref{form24} already gives an upper bound smaller than \eqref{incidence2}. Moreover, we claim that there exists a subset $\overline{\mathcal{T}} \subset \mathcal{T}$ with $|\overline{\mathcal{T}}| \approx |\mathcal{T}| \approx \delta^{-s}$ such that
\begin{equation}\label{form2} 
|\{p \in \mathcal{P} : x \in \ell \text{ for some } x \in P \cap p, \, \ell \in \mathcal{L} \cap T\}| \sim |\mathcal{I}(P,\mathcal{L} \cap T)|_{\delta} \approx \delta^{-s/2}, ~\forall T \in \overline{\mathcal{T}}. 
\end{equation}
(The middle "$\sim$" follows from Lemma \ref{lemma3} applied to $P$ and $\mathcal{L} \cap T = \{\ell \in \mathcal{L} : \ell \subset T\}$.) To see \eqref{form2}, we pigeonhole a subset $\overline{\mathcal{T}} \subset \mathcal{T}$ with the properties that $|\mathcal{I}(P,\mathcal{L} \cap T)|_{\delta}$ is roughly constant for any $T \in \overline{\mathcal{T}}$, say $|\mathcal{I}(P,\mathcal{L} \cap T)|_{\delta} \sim M_0$, and moreover
\begin{displaymath} 
|\mathcal{I}(P,\mathcal{L} \cap \overline{\mathcal{T}})|_{\delta} \approx |\mathcal{I}(P,\mathcal{L})|_{\delta} \approx \delta^{-3s/2}. 
\end{displaymath}
By \eqref{form24}, we have $|\mathcal{I}(P,\mathcal{L} \cap \overline{\mathcal{T}})| \lessapprox \delta^{-s}|\overline{\mathcal{T}}|^{1/2}$. Thus $|\overline{\mathcal{T}}| \approx \delta^{-s}$, and then $M_0 \cdot |\overline{\mathcal{T}}| \approx \delta^{-3s/2}$ implies that $M_0 \approx \delta^{-s/2}$.
In the sequel, we simplify notation by dropping the "bar" and denoting $\overline{\mathcal{T}}$ still by $\mathcal{T}$. (So, formally, the proof will finally produce a $(\delta,\delta^{u})$-clique inside $(P \cap \mathcal{P}_{j}) \times (\mathcal{L} \cap \overline{\mathcal{T}})$.) We have now reduced matters to a situation where $\mathcal{P} = \mathcal{D}_{\delta}(P)$ is $\{2^{-jH}\}_{j = 1}^{m}$-uniform with $H \sim_{\epsilon} 1$, and $\mathcal{T} = \mathcal{T}^{\delta}(\mathcal{L})$ satisfies \eqref{form2}. 
	
\subsection*{Step 2. Finding the dyadic scale $\Delta \in [\delta,1]$.} Fix $T \in \mathcal{T}$. We want to show, roughly speaking, that most of the squares in the set
\begin{displaymath} \overline{\mathcal{P}}_{T} := \{p \in \mathcal{P} : p \in \ell \text{ for some } x \in P \cap p, \, \ell \in \mathcal{L} \cap T\}, \end{displaymath}
familiar with \eqref{form2}, lies inside a single rectangle of dimensions roughly $\delta \times \sqrt{\delta}$. \emph{A priori} we have no information about the distribution of $\overline{\mathcal{P}}_{T}$, but at least Proposition \ref{prop2} applied to each individual $\overline{\mathcal{P}}_{T}$, $T \in \mathcal{T}$, allows us to find a few useful objects.
	
Namely, we apply Proposition \ref{prop2} with $d=2$, $\mathbf{C}:= 400/(su)$ and $t := \tfrac{s}{2}$ (the point being that $|\overline{\mathcal{P}}_{T}| \approx \delta^{-s/2}$). Let 
\begin{equation}\label{def:eta0} \eta_{0} = \eta_{0}(\mathbf{C},s) \leq \frac{2}{\mathbf{C}} \leq \frac{su}{200} \end{equation}
be the constant provided by Proposition \ref{prop2}. With this notation established, we can state the sufficient condition for the constant "$\epsilon$" in Theorem \ref{main2}:
\begin{equation}\label{defEpsilon} 0 < \epsilon \leq c\eta_{0}, \end{equation}
where $c = c(s,u) > 0$ is a small constant depending only on $s,u$, determined later.
	
The conclusion of Proposition \ref{prop2} is that for every $T \in \mathcal{T}$, the following objects exist:
\begin{itemize}
\item[(i) \phantomsection \label{i}]  a number $\eta_{T} \in [\eta_{0},2/\mathbf{C}] \subset [\eta_{0},su/200]$ and a scale $\Delta_{T} \in [\delta,1]$;
\item[(ii) \phantomsection \label{ii}] a square $Q_{T} \in \mathcal{D}_{\Delta_{T}}(\overline{\mathcal{P}}_{T})$ and a subset $\mathcal{P}_{T} \subset \overline{\mathcal{P}}_{T} \cap Q_{T}$.
\end{itemize} 
These objects satisfy
\begin{itemize}
\item[(P1)] \phantomsection \label{P1}$|\mathcal{P}_{T}| \geq \delta^{\eta_{T}}|\overline{\mathcal{P}}_{T}| \gtrapprox \delta^{\eta_{T} - s/2}$ (using \eqref{form2}).
\item[(P2)] \phantomsection \label{P2}$S_{Q_{T}}(\mathcal{P}_{T})$ is a $((\delta/\Delta_{T}),\mathbf{C}\eta_T,O_{\mathbf{C},s}(1))$-set. 
\end{itemize}
The parameters $\Delta_{T},\eta_{T}$ initially depend on "$T$", but this can be fixed by another pigeonholing. Indeed, there exist a subset $\mathcal{T}' \subset \mathcal{T}$ of cardinality $|\mathcal{T}'| \gtrsim (\log \tfrac{1}{\delta})^{-1} |\mathcal{T}|$, and fixed numbers $\eta \in [\eta_{0},2/\mathbf{C}]$ and $\Delta \in [\delta,1] \cap \{2^{-jH}\}_{j = 1}^{m}$ such that 
\begin{displaymath} 
\eta_{T} \in [\eta,2\eta] \quad \text{and} \quad \Delta_{T} \in [\Delta,2^{H}\Delta], \quad \forall T \in \mathcal{T}'. 
\end{displaymath}
Since there will be no difference between $\mathcal{T}'$ and $\mathcal{T}$ for the remainder of the argument, we simplify notation by denoting $\mathcal{T}'$ again by $\mathcal{T}$. In other words, we assume that $\eta_{T} \sim \eta$ and $\Delta_{T} \sim_{\epsilon} \Delta$ for all $T \in \mathcal{T}$.

Now we arrive at a key claim in the proof: assuming \eqref{def:eta0}, \eqref{defEpsilon} and $\delta> 0$ sufficiently small in terms of $s, u$, we have
\begin{equation}\label{mainClaim} 
\delta^{1/2 + 2\eta/s} \lessapprox \Delta \leq \delta^{1/2 - u/3}. 
\end{equation}
Once \eqref{mainClaim} has been established, it will be a simple matter to find a $(\delta,\delta^u)$-clique in \textbf{Step 3.} The rest of \textbf{Step 2.} is devoted to proving \eqref{mainClaim}.

By the Katz-Tao $(\delta, s, \delta^{-\epsilon})$-set condition of $P$ (hence $\mathcal{P}$) and the lower bound \nref{P1} for $|\mathcal{P}_T|$, we have
\begin{displaymath}
(\tfrac{\Delta}{\delta})^s\gtrapprox |\mathcal{P}\cap Q_T|\geq |\mathcal{P}_T|\gtrapprox \delta^{2\eta-s/2}.
\end{displaymath}
This yields the lower bound in \eqref{mainClaim} for $\Delta$.

Before showing the upper bound, we record that the sets $S_{Q_{T}}(\mathcal{P}_{T})$ satisfy a $2$-ends condition of the form \eqref{form11}. To verify this, it suffices to use the $((\delta/\Delta_{T}),\mathbf{C}\eta_T,O_{\mathbf{C},s}(1))$-set condition of $S_{Q}(\mathcal{P}_{T})$:
\begin{equation*}
|S_{Q_T}(\mathcal{P}_{T}) \cap B(y,r)| \leq O_{\mathbf{C},s}(1) r^{\mathbf{C}\eta}|S_{Q_T}(\mathcal{P}_{T})|, \quad \forall y \in \R^{2}, \, T \in \mathcal{T},
\end{equation*}
where $r\in[\delta/\Delta_{T},1]$. Let $r_{0} := (3^{-1}O_{\mathbf{C},s}(1))^{1/(\mathbf{C}\eta)}$, a constant depending only on $s,u$. Note that $r_{0} \geq \delta/\Delta_{T}$ by the lower bound in \eqref{mainClaim}, provided $\delta > 0$ is sufficiently small. Then
\begin{equation}\label{2-ends condition 1}
|S_{Q_T}(\mathcal{P}_{T}) \cap B(y,r_0)| \leq \tfrac{1}{3}|S_{Q_T}(\mathcal{P}_{T})|, \quad \forall y \in \R^{2}, \, T \in \mathcal{T}. 
\end{equation}
Thanks to \eqref{2-ends condition 1}, we are able to use the estimate \eqref{form12} to prove the upper bound for $\Delta$.

As an intermediate goal, we want to show that
\begin{equation}\label{form4} 
|\mathcal{P}|_{\Delta} \lessapprox \delta^{-4\eta}\Delta^{-s}. 
\end{equation}
Since $\mathcal{P}$ is $\{2^{-jH}\}_{j = 1}^{m}$-uniform, and $\Delta \in \{2^{-jH}\}_{j = 1}^{m}$, there exists a constant $M \lessapprox (\Delta/\delta)^s$ such that $|\mathcal{P} \cap Q| = M$ for all $Q \in \mathcal{D}_{\Delta}(\mathcal{P})$, and in fact $M = |\mathcal{P}|/|\mathcal{P}|_{\Delta}$. We claim that $M \gtrapprox \delta^{4\eta}(\Delta/\delta)^s$, which will prove \eqref{form4} because $|\mathcal{P}|_{\Delta} \approx \delta^{-s}/M$.
	
Recall the squares $Q_{T} \in \mathcal{D}_{\Delta}(\mathcal{P})$, $T \in \mathcal{T}$, in \nref{ii}. By the pigeonhole principle, there exists at least one $Q_{0} \in \mathcal{D}_{\Delta}(\mathcal{P})$ such that
\begin{equation}\label{form6} 
|\{T \in \mathcal{T} : Q_{T} = Q_{0}\}| \geq \frac{|\mathcal{T}|}{|\mathcal{P}|_{\Delta}} \approx \frac{|\mathcal{P}|}{|\mathcal{P}|_{\Delta}} = M. 
\end{equation}
The square $Q_{0}$ will be fixed for the rest of the proof, and we write
\begin{equation}\label{form28}
\mathcal{T}_{0} := \{T \in \mathcal{T} : Q_{T} = Q_{0}\}.
\end{equation}
Notice that whenever $Q_{T} = Q_{0}$, since $\cup \mathcal{P}_{T} \subset Q_{0}$,  we have 
\begin{displaymath} 
|\{p \in \mathcal{P} \cap Q_{0} : x \in \ell \text{ for some } x \in P \cap p, \, \ell \in \mathcal{L} \cap T\}| \geq |\mathcal{P}_{T}| \gtrapprox \delta^{2\eta -s/2}. 
\end{displaymath}
This will be important in establishing the $\delta^{u}$-clique property in \textbf{Step 3.}
	
For $T_1, T_2\in \mathcal{T}_0$, we say that two intersections $T_{1}\cap Q_{0}$ and $T_{2}\cap Q_{0}$ are \emph{comparable} if there exists a rectangle of dimensions $\sim (\Delta \times \delta)$ containing both $T_{j}\cap Q_{0}$, $j \in \{1,2\}$. (The exact requirement for the dimensions of $R$ is determined by the following: if $T_{1} \cap Q_{0}$ and $T_{2} \cap Q_{0}$ are incomparable, then the rescaled sets $S_{Q_{0}}(T_{j} \cap Q_{0})$ are contained in distinct ordinary $C(\delta/\Delta)$-tubes. This will be used when we soon apply Proposition \ref{prop3}.)
	
	We claim that the the family
	\[\mathcal{T}_0\cap Q_0:=\{T\cap Q_0: T\in\mathcal{T}_0\}\]
	contains $\gtrapprox \Delta^{s}M$ incomparable intersections. This is based on the Katz-Tao $(\delta,s)$-condition of $\mathcal{L}$ (hence $\mathcal{T}$). Assume that $\{T_{1},\ldots,T_k\}$ is a family of dyadic $\delta$-tubes such that every intersection $\{T_{i} \cap Q_{0}\}$ is comparable to one fixed intersection $\{T_{i_{0}} \cap Q_{0}\}$. Let $l_i$ be some line contained in $T_i$. Then the angle between any two lines of $\{l_1, \ldots, l_k\}$ is $\lesssim \delta/\Delta$, thus there exists a $(2 \times (C\delta/\Delta))$-rectangle $\mathbf{T}$ such that $[0,1]^{2} \cap T_i \subset \mathbf{T}$ for $1 \leq i \leq k$,
	where $C \geq 1$ is absolute. Since $\mathcal{T}$ is a Katz-Tao $(\delta, s, \delta^{-\epsilon})$-set, it follows 
	\begin{equation}\label{form25} k \lessapprox \Delta^{-s}. \end{equation}
	From this and \eqref{form6}, we deduce that the family $\mathcal{T}_0 \cap Q_{0}$ has $\gtrapprox \Delta^s M$ incomparable intersections, as desired. We denote this in a slightly \emph{ad hoc} manner as 
	\begin{displaymath} |\mathcal{T}_0 \cap Q_{0}|_{\Delta \times \delta} \gtrapprox \Delta^{s} M. \end{displaymath}
	This enables us to use Proposition \ref{prop3} after rescaling. Indeed, if $T_1\cap Q_0$ and $T_2\cap Q_0$ are incomparable, then the rescaled sets $S_{Q_0}(T_1\cap Q_0)$ and $S_{Q_0}(T_2\cap Q_0)$ are distinct $(\delta/\Delta)$-tubes. Since $S_{Q_T}(\mathcal{P}_{T})$ satisfies the $2$-ends condition by \eqref{2-ends condition 1}, we infer from Proposition \ref{prop3} applied at scale $\delta/\Delta$ that
	\begin{equation}\label{form7} 
		\begin{split}
			M &= |\mathcal{P} \cap Q_{0}| = |S_{Q_0}(\mathcal{P} \cap Q_{0})|\\
			&\gtrapprox |\mathcal{T} \cap Q_{0}|_{\Delta \times \delta}^{1/2} \cdot \delta^{2\eta - s/2}\gtrapprox \delta^{2\eta}\left(\tfrac{\Delta}{\delta}\right)^{s/2}M^{1/2}.
		\end{split}
	\end{equation}
	Rearranging this inequality leads to $M \gtrapprox \delta^{4\eta}(\Delta/\delta)^s$. This finally proves \eqref{form4}. 
	
	We can now deduce from \eqref{form6} that
	\begin{equation}\label{form14} 
		|\mathcal{T}_{0}| \gtrapprox \delta^{4\eta}\left(\frac{\Delta}{\delta}\right)^s. 
	\end{equation}
	
	Next, define a subset $\Xi \subset \mathcal{T}_{0}$ to be a \emph{tube packet} if $\Xi$ has the form 
	\begin{displaymath} \Xi = \{T \in \mathcal{T}_{0} : T \cap Q_{0} \subset R\}, \end{displaymath}
	where $R$ is a rectangle of dimensions $\sim (\delta \times \Delta)$. Thus, the intersections $T \cap Q_{0}$, $T \in \Xi$, are pairwise comparable. In \eqref{form25}, we showed that the cardinality of every tube packet $\Xi$ satisfies $|\Xi|\lessapprox \Delta^{-s}$. By the pigeonhole principle, there exists a value $n \in \{1,\ldots,\lessapprox \Delta^{-s}\}$ such that $\approx |\mathcal{T}_{0}|$ tubes of $\mathcal{T}_{0}$ are contained in tube packets $\Xi_{1},\ldots,\Xi_{L}$ with $|\Xi_{j}| \sim n$. Since $n \lessapprox \Delta^{-s}$, we have the lower bound
	\begin{equation}\label{form26}                                                                                                                                         
		L \approx |\mathcal{T}_{0}|/n \stackrel{\eqref{form14}}{\gtrapprox} \delta^{4\eta}(\Delta/\delta)^s \cdot \Delta^s = \delta^{4\eta}(\Delta^{2}/\delta)^s.   
	\end{equation}
	On the other hand, we can match this with an upper bound by repeating the argument of \eqref{form7} (and using $|\mathcal{T} \cap Q_{0}|_{\Delta \times \delta} \geq L$):
	\begin{equation}\label{form21} 
		\left(\frac{\Delta}{\delta}\right)^s \gtrapprox M = |\mathcal{P} \cap Q_{0}| \gtrapprox L^{1/2} \cdot \delta^{2\eta -s/2}. 
	\end{equation}
	Rearranging this gives $L \lessapprox \delta^{-4\eta}(\Delta^{2}/\delta)^s$. Hence we get $\delta^{4\eta}(\Delta^{2}/\delta)^s \lessapprox L \lessapprox \delta^{-4\eta}(\Delta^{2}/\delta)^s$.
	It will also be useful to record that
	\begin{equation}\label{form15} 
		n \approx \frac{|\mathcal{T}_{0}|}{L} \stackrel{\eqref{form14}}{\gtrapprox} \frac{\delta^{4\eta}(\Delta/\delta)^s}{(\Delta^{2}/\delta)^s \cdot \delta^{-4\eta}} = \delta^{8\eta} \cdot \Delta^{-s}. 
	\end{equation}
	
	Next, for each tube packet $\Xi_{j}$, choose one representative $T_{j} \in \Xi_{j}$; we may assume that the intersections $T_{j} \cap Q_{0}$ are incomparable for different indices "$j$" (there exist $\sim L$ packets such that their representatives have this property, and we restrict attention to those packets without changing notation). Let $S_{Q_{0}} \colon Q_{0} \to [0,1)^{2}$ be the rescaling map, and define the following rescaled sets:
	\begin{itemize}
		\item $\mathbb{T} := \{S_{Q_{0}}(T_{j} \cap Q_{0}) : 1 \leq j \leq L\}$;
		\item $\mathbb{P} := \{S_{Q_{0}}(p) : p \in \mathcal{P} \cap Q_{0}\} \subset \mathcal{D}_{\delta/\Delta}$;
		\item $\mathbb{P}_{T} := S_{Q_{0}}(\mathcal{P}_T)$, where $\mathcal{P}_T \subset \mathcal{P} \cap T \cap Q_{T}$ is the subset obtained in \nref{ii}.
	\end{itemize}
	Here $\mathbb{T}$ is a collection of distinct $(\delta/\Delta)$-tubes with $|\mathbb{T}| = L$. We also recall from \nref{P1} that $|\mathbb{P}_{T}| \gtrapprox \delta^{2\eta - s/2}$.
	
	We are about to apply Theorem \ref{thm:Furstenberg-minimal-non-concentration} to the objects $\mathbb{T}$ and $\mathbb{P}_{T}$. We will leave to the reader the small technical point that the sets in $\mathbb{T}$ are not exactly (ordinary) $(\delta/\Delta)$-tubes. Each element of $\mathbb{T}$ is, however, contained in some $C(\delta/\Delta)$-tube. Theorem \ref{thm:Furstenberg-minimal-non-concentration} can then be applied to a maximal distinct subset in the ensuing family of $C(\delta/\Delta)$-tubes.
	
	The main challenge in applying Theorem \ref{thm:Furstenberg-minimal-non-concentration} is to verify the non-concentration condition \eqref{form10} for the collection of $(\delta/\Delta)$-tubes $\mathbb{T}$. This amounts to checking the following. Let
	\begin{equation}\label{def:rho} 
		\rho := \left(\frac{\delta}{\Delta}\right)|\mathbb{T}|^{1/2} \approx \left(\frac{\delta}{\Delta} \right) \cdot \left(\frac{\Delta^{2}}{\delta} \right)^{s/2} \cdot \delta^{\pm 4\eta} = \Delta^{s-1}\delta^{1-s/2 \pm 4\eta}. 
	\end{equation}
	Let $T_{\rho}$ be an arbitrary $(2 \times \rho)$-rectangle, and consider the quantity
	\begin{equation}\label{form16} 
		X := |\{\mathbf{T} \in \mathbb{T} : \mathbf{T} \subset T_{\rho}\}|. 
	\end{equation}
	After rescaling back to $Q_{0}$, there exist $X$ indices "$j$" such that $T_{j} \cap Q_{0} \subset \bar{T}_{\rho}$, where $\bar{T}_{\rho}:=S_{Q_{0}}^{-1}(T_{\rho})$ is now a rectangle of dimensions $(2\Delta \times \rho \Delta)$. A little trigonometry shows that whenever $T_{j} \cap Q_{0} \subset \bar{T}_{\rho}$, then all the tubes $T \in \Xi_{j}$ in the packet represented by $T_{j}$ satisfy
	\begin{displaymath} 
		T \cap [0,1]^{2} \subset [\bar{T}_{\rho}]_{A\rho}, 
	\end{displaymath} 
	where $[\bar{T}_{\rho}]_{A\rho}$ refers to the $A\rho$-neighbourhood of $\bar{T}_{\rho}$, and $A \geq 1$ is an absolute constant. Now, recalling from \eqref{form15} that $|\Xi_{j}| \sim n \gtrapprox \delta^{8\eta} \Delta^{-s}$, we infer from the Katz-Tao $(\delta,s,\delta^{-\epsilon})$-set condition of $\mathcal{T}$ that
	\begin{displaymath} 
		X \cdot \Delta^{-s} \cdot \delta^{8\eta} \lessapprox |\{T \in \mathcal{T} : T \subset [\bar{T}_{\rho}]_{A\rho}\}| \lessapprox \left(\frac{\rho}{\delta} \right)^s \stackrel{\eqref{def:rho}}{\lessapprox} (\Delta^{s-1}\delta^{-s/2 - 4\eta})^s.
	\end{displaymath} 
	This implies $X \lessapprox \delta^{-8\eta-4s\eta} (\Delta^s \cdot \delta^{-s/2})^s$. Since $|\mathbb{T}| = L \gtrapprox \delta^{4\eta} (\Delta^{2}/\delta)^s$ by \eqref{form26}, we see that
	\begin{displaymath} 
		|\{\mathbf{T} \in \mathbb{T} : \mathbf{T} \subset T_{\rho}\}| = X \lessapprox \delta^{-16\eta} \left(\frac{\delta}{\Delta^2} \right)^{s(1-s/2)}  |\mathbb{T}|. 
	\end{displaymath}
	We claim that this implies the upper bound $\Delta \leq \delta^{1/2 - u/3}$ asserted in \eqref{mainClaim}. 
	
	Assume that this fails: thus $\Delta > \delta^{1/2 - u/3}$. Then $\delta^{-16\eta}(\delta/\Delta^2)^{s(1-s/2)}<\delta^{su/5}$ thanks to $\eta<su/200$. Therefore, the non-concentration condition of \eqref{form10} is satisfied with exponent "$su/5$" in place of "$u$". Recalling that the sets $\mathbb{P}_{T}$ are $(\delta/\Delta,\mathbf{C}\eta,O_{\mathbf{C},s}(1))$-sets with $|\mathbb{P}_{T}| \gtrapprox \delta^{-s/2 + 2\eta} =: N$, Theorem \ref{thm:Furstenberg-minimal-non-concentration} (with $\alpha = \tfrac{1}{5} \cdot (\mathbf{C}\eta)su/8$) implies the following improvement over the $2$-ends bound in \eqref{form7}:
	\begin{displaymath}
		\begin{split}
			\left(\frac{\Delta}{\delta}\right)^s &\gtrapprox |\mathbb{P}| \geq N \cdot |\mathbb{T}|^{1/2}\cdot (\delta/\Delta)^{-\mathbf{C}\eta \cdot su/40} \\
			&\stackrel{\eqref{form26}}{\gtrapprox} \delta^{4\eta-s/2} \cdot \left(\frac{\Delta^{2}}{\delta} \right)^{s/2} \cdot (\delta/\Delta)^{-\mathbf{C}\eta \cdot su/40}. 
		\end{split}
	\end{displaymath}
	Rearranging the inequality, using $\Delta>\delta^{1/2-u/3} \geq \delta^{1/2}$, and recalling from above \eqref{def:eta0} that $\mathbf{C} = 400/(su)$, we obtain
	\begin{displaymath}
		1\geq C\delta^{-\tfrac{\mathbf{C}\eta \cdot su}{80}+4\eta+C\epsilon} = C\delta^{-\eta + C\epsilon}. \end{displaymath}
	Here $C =C(s,u) > 0$ is a constant depending only on $s$ and $u$. Since $\epsilon \leq c\eta_{0} \leq \eta/(2C)$ by \eqref{defEpsilon} and \nref{i}, we get a contradiction for all $\delta > 0$ sufficiently small. This concludes the proof of the upper bound in \eqref{mainClaim}.
	
	\subsection*{Step 3. Finding a $(\delta,\delta^u)$-clique.}
	Recalling \eqref{form15}, note that there exists at least one tube packet, denoted $\Xi_{0} = \{T \in \mathcal{T}_{0} : T \cap Q_{0} \subset R_{\Delta}\}$, such that
	\begin{displaymath} |\Xi_{0}| \sim n \gtrapprox \delta^{8\eta} \cdot \Delta^{-s} \geq \delta^{-s/2 + su/3+ 8\eta}. \end{displaymath} 
	Here $R_{\Delta}$ is a rectangle of dimensions $\sim (\delta \times \Delta)$. 
	Since $\Xi_{0} \subset \mathcal{T}_{0}$ (recall the definition of $\mathcal{T}_{0}$ from \eqref{form28}), $|\{p \in \mathcal{P} \cap Q_{0} : x \in \ell \text{ for some } p \in P \cap p, \, \ell \in \mathcal{L} \cap T\}| \gtrapprox \delta^{-s/2 + 2\eta}$ for all $T \in \Xi_{0}$. Now we define
	\begin{equation}\label{clique}
		P':= P \cap Q_{0} \quad \text{and} \quad \mathcal{L}':= \mathcal{L} \cap \Xi_{0} := \{\ell \in \mathcal{L} : \ell \subset T \text{ for some } T \in \Xi_{0}\}.
	\end{equation}
	Then $|\mathcal{I}(P',\mathcal{L}')|_{\delta} \gtrapprox |\Xi_{0}|\delta^{-s/2 + 2\eta} \gtrapprox \delta^{-s + su/3 + 10\eta}$. On the other hand, using the Katz-Tao $(\delta,s,\delta^{-\epsilon})$-property of $P$, the upper bound $|\mathcal{L}'|_{\delta} = |\Xi_{0}| \lessapprox \Delta^{-s}$ valid for all tube packets, and crucially both inequalities in \eqref{mainClaim},
	\[\delta^{-s/2+3\eta}\leq |P'|_{\delta} \leq \delta^{-s/2-su/3-\eta},\quad \delta^{-s/2+su/3+10\eta}\leq |\mathcal{L}'|_{\delta} \leq \delta^{-s/2-3\eta}.\]
	Since $\eta\leq us/200$, it follows from the numerology above that
	\begin{equation*}
		|\mathcal{I}(P',\mathcal{L}')|_{\delta} \geq \delta^{u}|P'|_{\delta}|\mathcal{L}'|_{\delta} ~~ \text{with}~~ |P'|_{\delta} \geq \delta^{u -s/2}, ~~|\mathcal{L}'|_{\delta} \geq \delta^{u -s/2}.
	\end{equation*}
	Thus, $P' \times \mathcal{L}'$ satisfies the claims of Theorem \ref{main2}. \end{proof}

\subsection{Case $s < t$}\label{subsection 3.2} 
The proof is similar to the case $s = t$, except that the argument does not rely on Theorem \ref{thm:Furstenberg-minimal-non-concentration}: in the variant of \textbf{Step 2.} below, a completely elementary argument gives the desired upper bound for $\Delta$. Where the proof is virtually the same as in the case $s = t$, we will omit some repeated details.

From now on, we fix the parameters $u \in (0,1]$ and $s, t\in(0,1]$ with $s<t$, as in the statement of Theorem \ref{main}. Recall that $f(s,t)= (s^2+st+t^{2})/(s+t)$.

\begin{proof}[Proof of Theorem \ref{main} in the case $s < t$]
	We use $A \lessapprox B$ to signify that there exists a constant $C > 0$, depending only on $s,t,u$, such that $A \leq C\delta^{-C\epsilon}B$. Here $\epsilon > 0$ is the constant from the main hypothesis \eqref{form36}. The constant $\epsilon > 0$ will be specified at \eqref{defEpsilon2}. Just like in the case $s = t$, it is easy to reduce matters to the situation where the slopes of the lines in $\mathcal{L}$ lie in $[-1,1]$. This makes Lemma \ref{lemma3} applicable.
	
	\subsection*{Step 1. Reduction to the case where $\mathcal{D}_\delta(P)$ is uniform} We also denote $\mathcal{P}:=\mathcal{D}_\delta(P)$ and $\mathcal{T} := \mathcal{T}^{\delta}(\mathcal{L})$. Then it suffices to prove Theorem \ref{main} under the following additional hypotheses:
	\begin{itemize}
		\item[(i)] $\mathcal{P}$ is $\{2^{-jH}\}_{j = 1}^{m}$-uniform for some $H \sim_{\epsilon} 1$.
		\item[(ii)] $|\mathcal{P}| \approx \delta^{-s}$ and $|\mathcal{T}| \approx \delta^{-t}$.
		\item[(iii)] $|\mathcal{I}(P,\mathcal{L} \cap T)|_{\delta} \approx \delta^{-s^2/(s+t)}, ~\forall T \in \mathcal{T}$, ~~where by Lemma \ref{lemma3}
		\[|\mathcal{I}(P,\mathcal{L} \cap T)|_{\delta} \sim |\{p \in \mathcal{P} : x \in \ell \text{ for some } x \in P \cap p, \, \ell \in \mathcal{L} \cap T\}|.\]
	\end{itemize}
	This reduction was carried out in detail in \textbf{Step 1.} of the case $s = t$, and the arguments are exactly the same, up to changing the numerology, and applying the case $s < t$ of Fu and Ren's Theorem \ref{combinatorial bound}. Morally, (ii) follows from \eqref{form36}, because if either $|\mathcal{P}| \ll \delta^{-s}$ or $|\mathcal{T}| \ll \delta^{-t}$, then Theorem \ref{combinatorial bound} already gives an improvement over \eqref{form36}. Eventually, (iii) follows from (ii) and \eqref{form36} after another pigeonholing argument: morally but inaccurately, this is the computation $|\mathcal{I}(P,\mathcal{L} \cap T)|_{\delta} \approx |\mathcal{T}|^{-1}|\mathcal{I}(P,\mathcal{L})|_\delta \approx \delta^{t - f(s,t)} = \delta^{-s^{2}/(s + t)}$.

	\subsection*{Step 2. Finding the dyadic scale $\Delta \in [\delta,1]$.} 
	The argument in this step will initially resemble the case $s = t$ closely, but eventually Theorem \ref{thm:Furstenberg-minimal-non-concentration} will not be needed. For each $T\in \mathcal{T}$, we write
	\begin{equation}\label{form55} \overline{\mathcal{P}}_{T} := \{p \in \mathcal{P} : p \in \ell \text{ for some } x \in P \cap p, \, \ell \in \mathcal{L} \cap T\}. \end{equation}
	By property (iii) in \textbf{Step 1.}, $|\overline{\mathcal{P}}_{T}|\approx \delta^{-s^2/(s+t)}$. Let $d=2$, $\mathbf{C}:=160/(su(t-s))$ and apply Proposition \ref{prop2} for each $\overline{\mathcal{P}}_{T}$ with parameters $(d,\mathbf{C}, s^2/(s+t))$. Let 
	\begin{equation}\label{def:eta02} \eta_{0} = \eta_{0}(\mathbf{C},s,t) \leq \frac{2}{\mathbf{C}} \leq \frac{su(t-s)}{80} \end{equation}
	be the constant provided by Proposition \ref{prop2}. We claim that the following bound suffices for the parameter "$\epsilon$" in \eqref{form36}:
	\begin{equation}\label{defEpsilon2} 0 < \epsilon \leq c\eta_{0}. \end{equation}
	Here $c = c(s,t,u) > 0$ is a small constant depending only on $s,t,u$, determined later.
	
	Then for each $T \in \mathcal{T}$, there exist:
	\begin{itemize}
		\item[(a) \phantomsection \label{a}] a number $\eta_{T} \in [\eta_{0},2/\mathbf{C}]$ and a scale $\Delta_{T} \in [\delta,1]$;
		\item[(b) \phantomsection \label{b}] a square $Q_{T} \in \mathcal{D}_{\Delta_{T}}(\overline{\mathcal{P}}_{T})$ and a subset $\mathcal{P}_{T} \subset \overline{\mathcal{P}}_{T} \cap Q_{T}$;
	\end{itemize} 
	which satisfy
	\begin{itemize}
		\item[(L1) \phantomsection \label{L1}] $|\mathcal{P}_{T}| \geq \delta^{\eta_{T}}|\overline{\mathcal{P}}_{T}| \gtrapprox \delta^{\eta_{T} - s^{2}/(s + t)}$;
		\item[(L2) \phantomsection \label{L2}] $S_{Q_{T}}(\mathcal{P}_{T})$ is a $((\delta/\Delta_{T}),\mathbf{C}\eta_T,O_{\mathbf{C},s,t}(1))$-set. 
	\end{itemize}
	To remove the dependence on $T$ for $\eta_T$ and $\Delta_T$, pigeonhole a subset $\mathcal{T}' \subset \mathcal{T}$ of cardinality $|\mathcal{T}'| \gtrsim (\log \tfrac{1}{\delta})^{-1} |\mathcal{T}|$, and fixed numbers $\eta \in [\eta_{0},2/\mathbf{C}]$ and $\Delta \in [\delta,1]\cap \{2^{-jH}\}_{j = 1}^{m}$ such that 
	\begin{displaymath} 
		\eta_{T} \in [\eta,2\eta] \quad \text{and} \quad \Delta_{T} \in [\Delta,2^{H}\Delta], \quad \forall T \in \mathcal{T}'. 
	\end{displaymath}
	In the following, we simplify notation by denoting $\mathcal{T}'$ again by $\mathcal{T}$. 
	
	The rest of \textbf{Step 2.} is devoted to proving the following claim: under our choices for $\mathbf{C},\epsilon$, if $\delta>0$ is sufficiently small in terms of $s,t,u$, then
	\begin{equation}\label{mainClaim2} 
		\delta^{t/(s + t) + 2\eta/s} \lessapprox \Delta \leq \delta^{t/(s + t) - u/8}. 
	\end{equation}
	By the Katz-Tao $(\delta, s, \delta^{-\epsilon})$-set property of $P$ and the lower bound on $|\mathcal{P}_T|$, we deduce
	\begin{displaymath}
		(\tfrac{\Delta}{\delta})^s\gtrapprox |\mathcal{P}\cap Q_T|\geq |\mathcal{P}_T|\gtrapprox \delta^{2\eta- s^{2}/(s + t)},
	\end{displaymath}
	which yields the lower bound in \eqref{mainClaim2}.

	We then proceed to prove the upper bound in \eqref{mainClaim2}. We first verify that $|S_{Q_T}(\mathcal{P}_T)|$ satisfies a 2-ends condition. Indeed, for any $x\in\mathbb R^2$ and $r\in[\delta/\Delta_T,1]$, we have from \nref{L2} 
	\[|S_{Q_T}(\mathcal{P}_T)\cap B(x,r)|\leq  O_{\mathbf{C},s,t,}(1) r^{\mathbf{C}\eta}|S_{Q_T}(\mathcal{P}_{T})|.\]
	In particular,
	\begin{equation}\label{form53} |S_{Q_T}(\mathcal{P}_T)\cap B(x,r_0)|< \tfrac{1}{3}|S_{Q_T}(\mathcal{P}_T)|, \qquad ~\forall x\in\mathbb R^2 \end{equation}
	for $r_{0} := (3^{-1}O_{\mathbf{C},s,t}(1))^{1/(\mathbf{C}\eta)}$, which is a constant depending only on $s,t,u$. In particular, a dependence on "$r_{0}$" is allowed in the $\lessapprox$ notation below.
	
	For any $Q\in\mathcal{D}_\Delta(\mathcal{P})$, we get from the $(\delta, s, \delta^{-\epsilon})$-set condition and uniformity of $\mathcal{P}$ that 
	\begin{displaymath}
		M:=\tfrac{|\mathcal{P}|}{|\mathcal{P}|_{\Delta}} = |\mathcal{P}\cap Q| \lessapprox \left(\tfrac{\Delta}{\delta}\right)^s.
	\end{displaymath}
	Recall $Q_{T} \in \mathcal{D}_{\Delta}(\overline{\mathcal{P}}_{T})$ in \nref{b}. By the pigeonhole principle, there exists $Q_{0} \in \mathcal{D}_{\Delta}(\mathcal{P})$ such that
	\begin{equation}\label{form6-s<t} 
		|\{T \in \mathcal{T} : Q_{T} = Q_{0}\}| \geq \frac{|\mathcal{T}|}{|{\mathcal{P}}|_{\Delta}} \approx M \delta^{s-t}. 
	\end{equation}
	(The last equation follows from (ii) in \textbf{Step 2.}) We also write 
	\begin{equation}\label{def T_0}
		\mathcal{T}_{0} := \{T \in \mathcal{T} : Q_{T} = Q_{0}\}.
	\end{equation}
	We record that whenever $T \in \mathcal{T}_{0}$, then $Q_{T} = Q_{0}$ by definition, and \nref{b} and \nref{L1} imply
	\begin{equation}\label{form54}
		 |\overline{\mathcal{P}}_{T} \cap Q_{0}| \geq |\mathcal{P}_{T}| \gtrapprox \delta^{2\eta - s^{2}/(s + t)}, \qquad T \in \mathcal{T}_{0}.
	\end{equation}
	As in the case $s = t$, we need to find a lower bound on the number of incomparable intersections in the family
	\[\mathcal{T}_{0}\cap Q_{0}:=\{T \cap Q_{0} : T \in \mathcal{T}_{0}\}.\]
	Recall that $T_{1} \cap Q_{0}$ and $T_{2} \cap Q_{0}$ are comparable if there exists a rectangle of dimensions $\sim (\delta \times \Delta)$ containing both $T_{j}\cap Q_{0}$, $j \in \{1,2\}$. Assume $\{T_{j}\}_{j=1}^n \subset \mathcal{T}$ is a family (a \emph{tube packet}) such that every intersection $\{T_{j} \cap Q_{0}\}$ is comparable to one fixed intersection $\{T_{j_{0}} \cap Q_{0}\}$. Then $n \lessapprox \Delta^{-t}$ by the $(\delta, t, \delta^{-\epsilon})$-set condition of $\mathcal{T}$, see the proof of \eqref{form25}. From this and \eqref{form6-s<t}, we deduce that the family $\mathcal{T}_{0}\cap Q_{0}$ has $L \gtrapprox \Delta^t M \delta^{s-t}$ incomparable elements. By applying Proposition \ref{prop3} after rescaling by $S_{Q_{0}}$, and the $2$-ends condition we established in \eqref{form53},
	\begin{equation}\label{form7-s<t} 
		M = |\mathcal{P} \cap Q_{0}| \gtrsim L^{1/2} \cdot \delta^{2\eta - s^{2}/(s + t)} \gtrapprox (\Delta^t M \delta^{s-t})^{1/2}\cdot \delta^{2\eta - s^{2}/(s + t)}, 
	\end{equation}
	which implies $M \gtrapprox \delta^{4\eta-2s^{2}/(s + t)+s-t}\Delta^t$, and consequently
	\begin{equation}\label{form33}
		\delta^{4\eta - (s^{2} + t^{2})/(s + t)}\Delta^{t} = \delta^{4\eta-2s^{2}/(s + t) +s-t}\Delta^t\lessapprox M \lessapprox (\tfrac{\Delta}{\delta})^s.
	\end{equation}
	Since $s < t$, we may infer that
	\begin{equation}\label{form32}
		\Delta\lessapprox \delta^{t/(s + t) -4\eta/(t-s)}.
	\end{equation}
	Finally, recall from \eqref{def:eta02} that $\max\{\eta,\eta_{0}\} \leq 2/\mathbf{C} \leq su(t - s)/80$. Recall also that the "$\lessapprox$" notation hides a constant of the form $C\delta^{-C\epsilon}$, where $C = C(s,t,u) > 0$. Recalling from \eqref{defEpsilon2} that $\epsilon \leq c\eta_{0}$, and finally taking $c := C(s,t,u)^{-1}$, we may deduce from \eqref{form32} that $\Delta\leq \delta^{t/(s + t) -u/8}$, provided that $\delta > 0$ is sufficiently small in terms of $s,t,u$. This completes the proof of \eqref{mainClaim2}.

	\subsection*{Step 3. Finding a $(\delta,\delta^{u})$-clique.}
	Recall definition \eqref{def T_0}, and from a combination of \eqref{form6-s<t} and \eqref{form33} we deduce $|\mathcal{T}_{0}| \gtrapprox M\delta^{s - t} \gtrapprox \delta^{4\eta - 2s^{2}/(s + t) + 2s - 2t}\Delta^{t}$. A subset $\Xi \subset \mathcal{T}_{0}$ is a \emph{tube packet} if $\Xi$ has the form 
	\begin{displaymath} \Xi = \{T \in \mathcal{T}_{0} : T \cap Q_{0} \subset R\}, \end{displaymath}
	where $R$ is a rectangle of dimensions $\sim (\delta \times \Delta)$. Thus, the intersections $T \cap Q_{0}$, $T \in \Xi$, are pairwise comparable. We claim that there exists a tube packet $\Xi_{0}$ with $|\Xi_{0}| \geq \delta^{-t^{2}/(s + t) + u/2}$ (this is roughly the extremal cardinality of a tube packet allowed by the Katz-Tao $(\delta,t)$-set property of $\mathcal{T} \subset \mathcal{T}^{\delta}(\mathcal{L})$).
	
	In \textbf{Step 2.} we already showed that every tube packet $\Xi \subset \mathcal{T}_{0}$ satisfies $|\Xi|\lessapprox \Delta^{-t}$. By the pigeonhole principle, there exists a value $n \in \{1,\ldots,\lessapprox \Delta^{-t}\}$ such that $\approx |\mathcal{T}_{0}|$ tubes of $\mathcal{T}_{0}$ are contained in tube packets $\Xi_{1},\ldots,\Xi_{L}$ with $|\Xi_{j}| \sim n$. To get an upper bound for $L$, we recall the ("2-ends") lower bound \eqref{form7-s<t}:
	\begin{equation*}
		\left(\tfrac{\Delta}{\delta}\right)^s \gtrapprox  |\mathcal{P} \cap Q_{0}| \gtrapprox L^{1/2} \cdot \delta^{2\eta - s^{2}/(s + t)}.
	\end{equation*}
	This implies $L \lessapprox \delta^{-4\eta}\Delta^{2s}\delta^{-2st/(s + t)}$, so
	\begin{equation*}
		n \approx \frac{|\mathcal{T}_{0}|}{L} \gtrapprox \frac{\delta^{4\eta - 2s^{2}/(s + t) + 2s - 2t}\Delta^{t}}{\delta^{-4\eta}\Delta^{2s}\delta^{-2st/(s + t)}} = \delta^{8\eta} \cdot \delta^{(2st - 2t^{2})/(s + t)}\Delta^{t-2s}.
	\end{equation*}
	Using finally $\delta^{t/(s + t) + 2\eta/s} \lessapprox \Delta \leq \delta^{t/(s + t) - u/8}$ (by \eqref{mainClaim2}), $\eta \leq 2/\mathbf{C}$ (by \eqref{def:eta02}) and $\epsilon \leq 2c/\mathbf{C}$ (by \eqref{defEpsilon2}), we obtain after a little algebra the desired inequality $n \geq \delta^{-t^{2}/(s + t)+u/2}$. In particular, there exists a tube packet $\Xi_{0}= \{T \in \mathcal{T}_{0} : T \cap Q_{0} \subset R_{\Delta}\}$ such that $|\Xi_{0}| \sim n \geq \delta^{-t^{2}/(s + t)+u/2}$. Here $R_{\Delta}$ is a rectangle of dimensions $\sim (\delta \times \Delta)$.
	
	As in the case $s = t$, we now define
	\begin{displaymath}
		P':= P \cap Q_{0} \quad \text{and} \quad \mathcal{L}':= \mathcal{L} \cap \Xi_{0} := \{\ell \in \mathcal{L} : \ell \subset T \text{ for some } T \in \Xi_{0}\}.
	\end{displaymath}
	As recorded in \eqref{form54}, $|\overline{\mathcal{P}}_T\cap Q_0|\gtrapprox \delta^{2\eta-s^2/(s+t)}$ for any $T\in\Xi_{0}\subset \mathcal{T}_0$. Recalling the definition of $\overline{\mathcal{P}}_{T}$ from \eqref{form55}, this implies 
	\begin{displaymath} |\mathcal{I}(P',\mathcal{L}')|_{\delta} \gtrapprox |\Xi_{0}|\delta^{2\eta - s^{2}/(s + t)} \gtrsim \delta^{2\eta + u/2 - (s^{2} + t^{2})/(s + t)}. \end{displaymath}
	It further follows from the Katz-Tao conditions of $P' \subset P$ and $\mathcal{L}' \subset \mathcal{L}$ that
	\[\delta^{-s^{2}/(s + t)+2\eta}\leq |P'|_{\delta}\leq \delta^{-s^{2}/(s + t)-u/6-\eta},\quad \delta^{-t^{2}/(s + t)+u/2}\leq |\mathcal{L}'|_{\delta}\leq \delta^{-t^{2}/(s + t)-u/40-2\eta}.\]
	Recalling that $\eta\leq su(t-s)/80$, we easily conclude
	\begin{equation*}
		|\mathcal{I}(P', \mathcal{L}')|_{\delta}\sim |\mathcal{D}_\delta(\mathcal{I}(P', \mathcal{L}'))|\geq \delta^{u}|P'|_{\delta}|\mathcal{L}'|_{\delta} ~~ \text{with}~~ |P'|_{\delta}\geq \delta^{u -\frac{s^2}{s+t}}, ~~|\mathcal{L}'|_{\delta}\geq \delta^{u -\frac{t^2}{s+t}}.
	\end{equation*}
	This means that $P' \times \mathcal{L}'$ satisfies the claims in Theorem \ref{main}. \end{proof}

\subsection{Case $s>t$}\label{s 3.3} This is a standard duality argument, but we record the details. Our assumptions are: $P \subset [0,1]^{2}$ is a Katz-Tao $(\delta,s,\delta^{-\epsilon})$-set, $\mathcal{L} \subset \mathcal{A}(2)$ is a Katz-Tao $(\delta,t,\delta^{-\epsilon})$-set, $|\mathcal{I}(P,\mathcal{L})|_{\delta} \geq \delta^{\epsilon - f(s,t)}$, and $s > t$.

We may assume that the slopes of the lines in $\mathcal{L}$ lie in $[-1,1]$, equivalently $\mathcal{L}\subset D([-1,1] \times \R)$. This is the same argument we already described at the beginning of Section \ref{subsection 3.1}. Assuming this, we infer from Lemma \ref{lemma3} that
\begin{equation}\label{form51} |\mathcal{D}_{\delta}(\mathcal{I}(P,\mathcal{L}))| \gtrsim \delta^{\epsilon - f(s,t)}. \end{equation}
As a second initial reduction, we may assume that all the lines in $\mathcal{L}$ cross the $y$-axis in $[-2,2]$. Indeed, other lines (with slopes in $[-1,1]$) do not contain points of $P \subset [0,1]^{2}$.

For any line $l_{a,b}: =\{(x,y)\in \mathbb R^2: y=ax+b\}$, define the map $D^\star$ by $D^\star(l_{a,b})=(-a,b)$. Then $D^{\star}$ is bi-Lipschitz on the subset of $\mathcal{A}(2)$ consisting of lines with slopes in $[-1,1]$, in particular on $\mathcal{L}$. We write
\begin{displaymath}
	D^\star(T)=\{(-a,b): (a,b)\in q\}, \quad T=D(q)\in \mathcal{T}^\delta. 
\end{displaymath}
Now we set
\[P^\star=D^\star(\mathcal{L}), \quad \mathcal{L}^\star=D(P).\]
From the bi-Lipschitz properties of $D,D^{\star}$, it follows that $P^\star$ is a $(\delta, t, \delta^{-\epsilon})$-set contained in $[-1,1]\times [-2,2]$, and $\mathcal{L}^\star$ is a $(\delta, s, \delta^{-\epsilon})$-set of lines with slopes in $[-1, 1]$ and $y$-intersects in $\{0\} \times [-1, 1]$.

We claim that 
\begin{equation}\label{form52} |\mathcal{I}(P^{\star},\mathcal{L}^{\star})|_{\delta} \gtrsim \delta^{\epsilon - f(s,t)}. \end{equation}
Indeed, if $(p, T)$ is a pair counted by the left hand side of \eqref{form51}, there exist $(x, y)\in P \cap p$ and a line $l_{a,b}\in T\cap \mathcal{L}$ such that $(x,y)\in l_{a,b}$. Thus $y=ax+b$, then $b=-ax+y$, which means $(-a, b)\in D(x,y)$. Since 
\[(-a,b)\in P^\star\cap D^\star(T), \quad D(x,y)\in \mathcal{L}^\star\cap D(p),\]
the pair $(D^\star(T), D(p)) \in \mathcal{D}_{\delta}(P^{\star}) \times \mathcal{T}^{\delta}(\mathcal{L}^{\star})$ lies in the set
\begin{displaymath} \{(q,T') \in \mathcal{D}_{\delta}(P^{\star}) \times \mathcal{T}^{\delta}(\mathcal{L}^{\star}) : y \in \ell' \text{ for some } y \in P^{\star} \cap q, \, \ell' \in \mathcal{L}^{\star} \cap T'\}. \end{displaymath}
The map $(T,p) \mapsto (D^{\star}(T),D(p))$ is injective, so the set above has the same cardinality as the set in \eqref{form51}. This proves $|\mathcal{I}(P^{\star},\mathcal{L}^{\star})|_{\delta} \gtrsim \delta^{\epsilon - f(s,t)}$.

Recalling that $s > t$, we have reduced our problem to the case treated in Section \ref{subsection 3.2}, where the Katz-Tao exponent of $P^{\star}$ (namely $t$) is strictly lower than than the Katz-Tao exponent of $L^{\star}$ (namely $s$), and moreover $|\mathcal{I}(P^\star, \mathcal{L}^\star)|_\delta \gtrsim \delta^{\epsilon-f(s,t)}$. By the result in Section \ref{subsection 3.2}, there exists a $(\delta, \delta^u)$-clique $(P^\star)' \times (\mathcal{L}^\star)' \subset P^\star \times \mathcal{L}^\star$ such that
\[|\mathcal{I}((P^\star)', (\mathcal{L}^\star)')|_\delta \geq \delta^{u} |(P^\star)'|_{\delta} |(\mathcal{L}^\star)'|_{\delta} \quad \text{with}\quad|(P^\star)'|_{\delta} \geq \delta^{u-\tfrac{t^2}{t+s}}, \quad |(\mathcal{L}^\star)'|_{\delta} \geq \delta^{u-\tfrac{s^2}{t+s}}.\]
We finally transform $(P^\star)' \times (\mathcal{L}^\star)'$ back to a subset of $P \times \mathcal{L}$ by setting
\[P':=D^{-1}(\mathcal{L}^\star),\quad \mathcal{L}':={D^\star}^{-1}(P^\star).\]
Now $P' \times \mathcal{L}' \subset P \times \mathcal{L}$ is a $(\delta, \delta^{2u})$-clique; this uses a similar argument as the proof of \eqref{form52}, where we showed that the transformations $D,D^{\star}$ roughly preserve the $\delta$-covering number of incidences, as well as the $\delta$-covering numbers of the sets $\mathcal{L}^{\star},P^{\star}$.

\section{Proofs of Corollary \ref{cor3} and Proposition \ref{prop4}}\label{s4}

We start by proving Proposition \ref{prop4}. Here is the statement again:

\begin{proposition}\label{prop4a} There exists an absolute constant $C \geq 1$ such that the following holds. Let $P \times \mathcal{L} \subset [0,1)^{2} \times \mathcal{A}(2)$ be a $(\delta,\theta)$-clique. Then, there exists a rectangle $R \subset \R^{2}$ of dimensions $C(\delta \times \Delta)$, where $\Delta \in [\delta,C]$, such that
\begin{equation}\label{form40a} |P \cap R|_{\delta} \gtrsim \theta^{2}|P|_{\delta} \quad \text{and} \quad |\{\ell \in \mathcal{L} : R \subset [\ell]_{C\delta}\}|_{\delta} \gtrapprox \theta^{4}|\mathcal{L}|_{\delta}. \end{equation}
In particular, if $P$ is a Katz-Tao $(\delta,s)$-set, and $\mathcal{L}$ is a Katz-Tao $(\delta,t)$-set, then $|P|_{\delta}^{t}|\mathcal{L}|_{\delta}^{s} \lessapprox \theta^{-6}\delta^{-st}$. \end{proposition}

\begin{proof}  Arguing as in Section \ref{section 3}, we may assume $\mathcal{L} \subset D([-1,1) \times \R)$. Denote $\mathcal{P} := \mathcal{D}_{\delta}(P)$ and $\mathcal{T} := \mathcal{T}^{\delta}(\mathcal{L})$. Since $(P,\mathcal{L})$ is a $(\delta, \delta^u)$-clique, we deduce by Lemma \ref{lemma3}, 
\begin{equation}\label{form49} 
\theta|\mathcal{P}||\mathcal{T}| \lesssim  |\{(p,T) \in \mathcal{P} \times \mathcal{T} : x \in \ell \text{ for some } x \in P \cap p, \, \ell \in T \cap \mathcal{L}\}|. 
\end{equation}
For each $p\in \mathcal{P}$ and $T \in \mathcal{T}$, write 
\begin{itemize}
\item $\mathcal{P}_{T} := \{p \in \mathcal{P} : x \in \ell \text{ for some } x \in p\cap P, \, \ell \in T \cap \mathcal{L}\}$,
\item $\mathcal{T}_{p} := \{T \in \mathcal{T} : x \in \ell \text{ for some } x \in p\cap P, \, \ell \in T \cap \mathcal{L}\}$.
\end{itemize}
We start by applying Cauchy-Schwarz:
\begin{displaymath} 
\theta |\mathcal{P}||\mathcal{T}| \stackrel{\eqref{form49}}{\lesssim} \sum_{p \in \mathcal{P}} |\mathcal{T}_{p}| \leq |\mathcal{P}|^{1/2} \Big(\sum_{T,T'} |\mathcal{P}_{T} \cap \mathcal{P}_{T'}| \Big)^{1/2}  \Longrightarrow \sum_{T,T'} |\mathcal{P}_{T} \cap \mathcal{P}_{T'}|\geq \theta^{2}|\mathcal{P}||\mathcal{T}|^{2}. 
\end{displaymath} 
Therefore, there exists $\geq \tfrac{1}{2}\theta^{2}|\mathcal{T}|^{2}$ pairs $(T,T') \in \mathcal{T} \times \mathcal{T}$ such that
\begin{displaymath} 
|\mathcal{P}_{T} \cap \mathcal{P}_{T'}| \geq \tfrac{1}{2}\theta^{2}|\mathcal{P}|. 
\end{displaymath}
In particular, we may fix $T_{0} \in \mathcal{T}$ such that the set
\begin{displaymath} 
\mathcal{T}_{0} := \{T \in \mathcal{T} : |\mathcal{P}_{T} \cap \mathcal{P}_{T_{0}}| \geq \tfrac{1}{2}\theta^{2}|\mathcal{P}|\} 
\end{displaymath}
has cardinality $|\mathcal{T}_{0}| \geq \tfrac{1}{2}\theta^{2}|\mathcal{T}|$. Let $\sigma_{0} := \sigma(T_{0}) \in [-1,1]$ be the slope of $T_{0}$ (Definition \ref{def:dyadicTube}). We further split $\mathcal{T}_{0}$ into $\sim \log(1/\delta)$ subsets such that $|\sigma(T) - \sigma_{0}| \in [\alpha,2\alpha]$ for all $T$ in a fixed subset. (For $\alpha = \delta$, the defining condition is $|\sigma(T) - \sigma_{0}| \leq 2\delta$). One subset has cardinality $\gtrapprox \theta^{2}|\mathcal{T}|$, and we keep denoting this subset $\mathcal{T}_{0}$. Note that 
\begin{displaymath} 
\diam([0,1)^{2} \cap T \cap T_{0}) \leq C(\delta/\alpha), \qquad \forall T \in \mathcal{T}_{0}. \end{displaymath}
	
Next, let $\mathcal{R} = \mathcal{R}(\alpha)$ be a boundedly overlapping cover of $[0,1]^{2} \cap T_{0}$ by rectangles of dimensions $C'\delta \times C'(\delta/\alpha)$, with the property that if $T \in \mathcal{T}_{0}$, then some element $R \in \mathcal{R}$ covers the intersection $[0,1]^{2} \cap T \cap T_{0}$ with further requirement
\begin{equation}\label{form50} 
	\cup (\mathcal{P}_{T} \cap \mathcal{P}_{T_{0}}) \subset R. 
\end{equation}
This can be done if we choose $C'\geq 1$ large enough. We say that a rectangle $R \in \mathcal{R}$ is \emph{good} if $|P \cap R|_{\delta} \geq \tfrac{1}{2}\theta^{2}|\mathcal{P}|$. By the bounded overlap of the family $\mathcal{R}$, there are $\lesssim \theta^{-2}$ good rectangles in $\mathcal{R}$. Moreover, if $R \in \mathcal{R}$ satisfies \eqref{form50} for some $T \in \mathcal{T}_{0}$, then $R$ is good, because in that case
\begin{displaymath} 
|P \cap R|_{\delta} \geq |\mathcal{P}_{T_{0}} \cap \mathcal{P}_{T}| \geq \tfrac{1}{2}\theta^{2}|\mathcal{P}|, \qquad \forall T \in \mathcal{T}_{0}. 
\end{displaymath}

Now, since there are $\lesssim \theta^{-2}$ good rectangles and $\gtrapprox \theta^{2}|\mathcal{T}|$ possible intersections $T \cap T_{0}$, there exists a good rectangle $R_{0} \in \mathcal{R}$ covering $\gtrapprox \theta^{4}|\mathcal{T}|$ intersections $T \cap T_{0}$ (in the sense \eqref{form50}). Whenever $R_{0}$ satisfies \eqref{form50}, we can choose a suitable constant $C''\geq1$ such that $R_{0} \subset [\ell]_{C''\delta}$ for any line $\ell \in \mathcal{L}$ contained in $T$ (and there are such lines, since $T \in \mathcal{T} = \mathcal{T}^{\delta}(\mathcal{L})$. Consequently,
\begin{displaymath} 
|P \cap R_{0}|_{\delta} \geq \tfrac{1}{2}\theta^{2}|\mathcal{P}| \sim \theta^{2}|P|_{\delta} \quad \text{and} \quad |\{\ell \in \mathcal{L} : R_{0} \subset [\ell]_{C''\delta}\}| \gtrapprox \theta^{4}|\mathcal{T}| \sim \theta^{4}|\mathcal{L}|_{\delta}. 
\end{displaymath}
This completes the proof of the proposition, except for the \emph{In particular...} part. The following are consequences of the Katz-Tao conditions:
\begin{displaymath} |P \cap R|_{\delta} \lesssim (\Delta/\delta)^{s} \quad \text{and} \quad |\{\ell \in \mathcal{L} : R \subset [\ell]_{C\delta}\}|_{\delta} \lesssim \Delta^{-t}.\end{displaymath}
Combining these with \eqref{form40a} we find $|P|_{\delta}^{t}|\mathcal{L}|_{\delta}^{s} \lessapprox \theta^{-6}(\Delta/\delta)^{st}\Delta^{-st} = \theta^{-6}\delta^{-st}$, as claimed.  \end{proof} 

Here is again the statement of Corollary \ref{cor3}:

\begin{cor}\label{cor3a} Under the hypotheses of Theorem \ref{main}, there exists a list
\begin{displaymath} (P_{1} \times \mathcal{L}_{1}),\ldots,(P_{n} \times \mathcal{L}_{n}) \subset P \times \mathcal{L} \end{displaymath}
of $(\delta,\delta^{u})$-cliques satisfying \eqref{form41}, with the sets $\mathcal{D}_{\delta}(P_{j})$ disjoint, and $\sum_{j} |\mathcal{I}(P_{j},\mathcal{L}_{j})|_{\delta} \geq \delta^{u - f(s,t)}$. \end{cor}

We will use the following observation:
\begin{lemma}\label{lemma4} Let $P \subset \R^{2}$ and $\mathcal{L} \subset \mathcal{A}(2)$. Assume that there exists a constant $M \geq 0$ such that
	\begin{displaymath} |\{\ell \in \mathcal{L} : x \in \ell \text{ for some } x \in P \cap p\}|_{\delta} \leq M, \qquad p \in \mathcal{D}_{\delta}(P). \end{displaymath}
	Then, $|\mathcal{I}(P,\mathcal{L})|_{\delta} \lesssim M|P|_{\delta}$. \end{lemma} 
\begin{proof} Note that $\mathcal{I}(P,\mathcal{L}) \subset \bigcup_{p \in \mathcal{D}_{\delta}(P)} \mathcal{I}(P \cap p,\mathcal{L})$, so
\begin{displaymath} |\mathcal{I}(P,\mathcal{L})|_{\delta} \leq \sum_{p \in \mathcal{D}_{\delta}(P)} |\mathcal{I}(P \cap p,\mathcal{L})|_{\delta}. \end{displaymath}
Here further $\mathcal{I}(P \cap p,\mathcal{L}) \subset p \times \{\ell \in \mathcal{L} : x \in \ell \text{ for some } x \in P \cap p\}$, so
\begin{displaymath} |\mathcal{I}(P \cap p,\mathcal{L})|_{\delta} \lesssim |\{\ell \in \mathcal{L} : x \in \ell \text{ for some } x \in P \cap p\}|_{\delta}. \end{displaymath}
Combining these inequalities gives the claim. \end{proof}

\begin{proof}[Proof of Corollary \ref{cor3a}] We may assume that every $x \in P$ is contained on at least one line from $\mathcal{L}$. This is because we may remove from $P$ the points for which this fails without affecting the hypothesis $|\mathcal{I}(P,\mathcal{L})|_{\delta} \geq \delta^{\epsilon - f(s,t)}$. 

Fix $u > 0$. The statement of Corollary \ref{cor3} only gets stronger for smaller values of $u > 0$, so we may assume that $0 < u \leq st/(s + t)$.

In this proof, the notation "$\lessapprox$" only hides constants of order $(\log(1/\delta))^{C}$. Let $c \in (0,\tfrac{1}{20}]$ be an absolute constant to be determined later. Let $\bar{\epsilon} = \bar{\epsilon}(s,t,cu) > 0$ be the constant provided by Theorem \ref{main} with parameters, $s,t,cu$, and let 
\begin{equation}\label{form46} \epsilon := \min\{u/(8C(s,t)),\bar{\epsilon}/2\}, \end{equation}
where $C(s,t) := \max\{(s + t)/s,100\}$. We now claim the conclusion of Corollary \ref{cor3a} holds if $P$ is a Katz-Tao $(\delta,s,\delta^{-\epsilon})$-set, $\mathcal{L}$ is a Katz-Tao $(\delta,t,\delta^{-\epsilon})$-set, $|\mathcal{I}(P,\mathcal{L})|_{\delta} \geq \delta^{\epsilon - f(s,t)}$, and finally $\delta > 0$ small enough, depending on $s,t,u$ (although we will not track the necessary smallness of $\delta$ explicitly).

We start by finding a subset $\bar{P} \subset P$ such that $|\mathcal{I}(\bar{P},\mathcal{L})|_{\delta} \gtrapprox |\mathcal{I}(P,\mathcal{L})|_{\delta}$, and moreover
\begin{equation}\label{form42} |\{\ell \in \mathcal{L} : x \in \ell \text{ for some } x \in \bar{P} \cap p\}|_{\delta} \leq \delta^{-u/4 -t^{2}/(s + t)}, \qquad p \in \mathcal{D}_{\delta}(\bar{P}). \end{equation}
To see how this is done, write $\mathcal{P} := \mathcal{D}_{\delta}(P)$. For every $p \in \mathcal{P}$, define the quantity
\begin{displaymath} M(p) := |\{\ell \in \mathcal{L} : x \in \ell \text{ for some } x \in P \cap p\}|_{\delta} \in \{1,\ldots,|\mathcal{L}|_{\delta}\}. \end{displaymath} 
Note that $M(p) \geq 1$ since every $p \in \mathcal{P}\cap P$ is contained in at least one line in $\mathcal{L}$ by the reduction in the first paragraph. Let 
\begin{displaymath} P_{M} := \{x \in P : M(p_x) \in [M,2M]\}, \end{displaymath}
where $p_x \in \mathcal{D}_{\delta}$ is the dyadic square containing $x$. Since the condition $M(p_x) \in [M,2M]$ only depends on the dyadic "$\delta$-parent" of $x$, one has $P_{M} = P \cap (\cup \mathcal{D}_{\delta}(P_{M}))$. Now, find a number $M \in \{1,\ldots,|\mathcal{L}|_{\delta}\}$ and a subset $\bar{P} := P_{M}$ such that $M(p_x) \sim M$ for all $x \in \bar{P}$, and
\begin{displaymath} |\mathcal{I}(P,\mathcal{L})|_{\delta} \approx |\mathcal{I}(\bar{P},\mathcal{L})|_{\delta} \gtrsim M|\bar{P}|_{\delta}. \end{displaymath}
By Lemma \ref{combinatorial bound}, and using the Katz-Tao conditions of $\bar{P}$ and $\mathcal{L}$,
\begin{displaymath} |\mathcal{I}(\bar{P},\mathcal{L})| \lesssim_{\epsilon} \delta^{-3\epsilon} \delta^{-st/(s + t)}|\bar{P}|_{\delta}^{s/(s + t)}|\mathcal{L}|_{\delta}^{t/(s + t)} \leq \delta^{-4\epsilon}\delta^{-(st + t^{2})/(s + t)}|\bar{P}|_{\delta}^{s/(s + t)}. \end{displaymath}
Chaining these inequalities, and using $|\mathcal{I}(P,\mathcal{L})|_{\delta} \geq \delta^{\epsilon - f(s,t)}$, we first infer $|\bar{P}|_{\delta} \gtrapprox \delta^{C(s,t)\epsilon - s}$, and next
\begin{displaymath} M \lessapprox \frac{|\mathcal{I}(P,\mathcal{L})|_{\delta}}{\delta^{C(s,t)\epsilon - s}} \lesssim_{\epsilon} \delta^{-C(s,t)\epsilon}\delta^{s - f(s,t)} = \delta^{-C(s,t)\epsilon}\delta^{-t^{2}/(s + t)}. \end{displaymath}
This proves \eqref{form42} by the choice of $\epsilon > 0$ in \eqref{form46}.

Noting that $|\mathcal{I}(\bar{P},\mathcal{L})|_{\delta} \geq \delta^{2\epsilon - f(s,t)} \geq \delta^{\bar{\epsilon} - f(s,t)}$, we may now apply Theorem \ref{main} to find our first $(\delta,\delta^{cu})$-clique $P_{1} \times \mathcal{L}_{1} \subset \bar{P} \times \mathcal{L}$ satisfying 
\begin{equation}\label{form47} |\bar{P}_{1}|_{\delta} \geq \delta^{cu-s^{2}/(s + t)} \quad \text{and} \quad |\mathcal{L}_{1}|_{\delta} \geq \delta^{cu-t^{2}/(s + t)}. \end{equation} 
Here the constant $c>0$ will be chosen later. We note that $\bar{P}_{1}$ can be selected to be of the form $\bar{P}_{1} = \bar{P} \cap (\cup \overline{\mathcal{P}}_{1})$, where $\overline{\mathcal{P}}_{1} \subset \mathcal{D}_{\delta}(\bar{P})$, recall Remark \ref{rem5}. 

We also need a matching upper bound for $|\bar{P}_{1}|_{\delta}$. One way might be to simply restrict $\bar{P}_{1}$ to a further subset, but one would have to be careful with preserving the $(\delta,\delta^{u})$-clique property. A more straightforward argument is to use the first estimate in \eqref{form40a}, combined with \eqref{form39} and the Katz-Tao condition of $\bar{P}_{1} \subset P$ to deduce that
\begin{displaymath} |\bar{P}_{1}|_{\delta} \lesssim \delta^{-2cu}|\bar{P}_{1} \cap R|_{\delta} \lesssim \delta^{-Ccu}\left(\frac{\delta^{t/(s + t)}}{\delta}\right)^{s} = \delta^{-u/4 -s^{2}/(s + t)}. \end{displaymath}
Above, $C \geq 1$ was an absolute constant arising from the estimate \eqref{form39} for the diameter or $R$, and then $c > 0$ was chosen so small that $cC \leq \tfrac{1}{4}$. In particular, recalling \eqref{form42}, and using Lemma \ref{lemma4},
\begin{equation}\label{form45} |\mathcal{I}(\bar{P}_{1},\mathcal{L})|_{\delta} \lesssim \delta^{-u/2 - (s^{2} + t^{2})/(s + t)}. \end{equation}
For $\delta > 0$ small enough, and since $u \leq st/(s + t)$, this upper bound is far smaller than $|\mathcal{I}(\bar{P},\mathcal{L})| \gtrapprox \delta^{\epsilon - f(s,t)}$. In particular, we still have $|\mathcal{I}(\bar{P} \, \setminus \, \bar{P}_{1},\mathcal{L})| \geq \delta^{2\epsilon - f(s,t)}$. Therefore, the previous reasoning can be repeated to the $(\delta,s,\delta^{-\epsilon})$-set $\bar{P} \, \setminus \, \bar{P}_{1}$ and $(\delta,t,\delta^{-\epsilon})$-set $\mathcal{T}$ to produce a second $\delta^{cu}$-clique $\bar{P}_{2} \times \mathcal{L}_{2} \subset (\bar{P} \, \setminus \, \bar{P}_{1}) \times \mathcal{L}$. Again, by Remark \ref{rem5}, the set $\bar{P}_{2}$ can be selected to be of the form $\bar{P}_{2} = (\bar{P} \, \setminus \, \bar{P}_{1}) \cap (\cup \overline{\mathcal{P}}_{2})$, where $\overline{\mathcal{P}}_{2} \subset \mathcal{D}_{\delta}(\bar{P} \, \setminus \, \bar{P}_{1})$. In particular, using the analogous structure of $\bar{P}_{1}$ discussed above, $\mathcal{D}_{\delta}(\bar{P}_{1}) \cap \mathcal{D}_{\delta}(\bar{P}_{2}) = \emptyset$.

How many times can this argument be repeated? For every $\delta^{cu}$-clique $\bar{P}_{j} \times \mathcal{L}_{j} \subset \bar{P} \times \mathcal{T}$, the estimate \eqref{form45} holds with "$1$" replaced by "$j$" (by the same argument). Therefore,
\begin{displaymath} |\mathcal{I}(\bar{P}_{1} \cup \ldots \cup \bar{P}_{n},\mathcal{L})|_{\delta} \lesssim n \cdot \delta^{-u/2 - (s^{2} + t^{2})/(s + t)}. \end{displaymath}
This upper bound remains smaller than $\tfrac{1}{2}|\mathcal{I}(\bar{P},\mathcal{L})| \gtrapprox \delta^{\epsilon - f(s,t)}$ as long as
\begin{displaymath} n \lesssim \delta^{u/2 + 2\epsilon - st/(s + t)}. \end{displaymath}
So, the argument can safely be repeated at least $n := \delta^{3u/4 - st/(s + t)}$ times. At this stage, by the $\delta^{cu}$-clique property of $\bar{P}_{j} \times \mathcal{L}_{j}$, and the covering number bounds \eqref{form47},
\begin{displaymath} \sum_{1 \leq j \leq n} |\mathcal{I}(\bar{P}_{j},\mathcal{L})|_{\delta} \geq n \cdot \delta^{cu} \cdot |\bar{P}_{j}|_{\delta}|\mathcal{L}_{j}|_{\delta} = \delta^{3u/4 + 3cu - f(s,t)}. \end{displaymath}
Since $c \leq \tfrac{1}{20}$, this completes the proof of the corollary. \end{proof}

\bibliographystyle{plain}
\bibliography{references}
	
\end{document}